\documentclass[12pt,leqno]{amsart}
\usepackage[top=2.5cm, bottom=2cm, left=2.5cm, right=2.5cm]{geometry}
\usepackage{myamsart}
\usepackage{cite}
\usepackage{amssymb}
\usepackage{amsmath}
\usepackage{amsthm}
\usepackage{mathtools}
\usepackage{bm}
\usepackage{subfigure}
\usepackage[pdftex]{graphicx}
\usepackage{here}
\usepackage[normalem]{ulem}
\usepackage{mhchem}

\DeclareMathOperator{\tr}{tr}
\DeclareMathOperator{\GC}{GC}
\DeclareMathOperator{\first}{I}
\DeclareMathOperator{\second}{II}
\DeclareMathOperator{\con}{conv}

\renewcommand{\epsilon}{\varepsilon}
\theoremstyle{definition}
\newtheorem{definition}{Definition}[section]

\newtheorem{re}[definition]{Remark}
\theoremstyle{plain}
\newtheorem{thm}[definition]{Theorem}

\newtheorem{lem}[definition]{Lemma}
\newtheorem{prop}[definition]{Proposition}

\makeatletter
\@addtoreset{equation}{section}

\makeatother
\allowdisplaybreaks
\newcommand{\R}{{\mathbb R}}

\newcommand{\N}{{\mathbb N}}
\newcommand{\E}{{\mathbb R}}
\newcommand{\MV}{{\mathcal V}}
\newcommand{\ME}{{\mathcal E}}
\newcommand{\MF}{{\mathcal F}}
\newcommand{\MC}{{\mathcal C}}
\newcommand{\MM}{{\mathcal M}}
\newcommand{\MS}{{\mathcal S}}
\renewcommand{\v}[1]{{\bm {#1}}}

\renewcommand{\emph}[1]{{\itshape{#1}}}
\title[Construction of continuum from a discrete surface]{Construction of continuum from a discrete surface \\
  by its iterated subdivisions}
\author{Motoko Kotani, Hisashi Naito and Chen Tao}
\address{
  M.~Kotani:
  Mathematical Institute,
  Tohoku University,
  Aoba, Sendai 980-8578, Japan
  and 
  AIMR
  Tohoku University,
  Aoba, Sendai 980-8577, Japan
}
\email{motoko.kotani.d3@tohoku.ac.jp}
\address{
  H.~Naito:
  Graduate School of Mathematics, 
  Nagoya University, 
  Chikusa, Nagoya 464-8602, Japan
}
\email{naito@math.nagoya-u.ac.jp}
\address{
  C.~Tao:
  School of Mathematics and Statistics, 
  Northeast Normal University, China
}
\email{taoc387@nenu.edu.cn}
\thanks{Authors (KM and HN) were partially supported by JSPS KAKENHI Grants Numbers
JP17H06465, JP17H06466, and JP19K03488.
KM was partially supported by JST, CRESTW Grand Number JPMJCR17J4, Japan.
}
\begin{document}
\begin{abstract}
 Given a trivalent graph in the 3-dimensional Euclidean space, 
 we call it a discrete surface because it has a tangent space at each vertex determined by its neighbor vertices.
 To abstract a continuum object hidden in the discrete surface,
 we introduce a subdivision method by applying the Goldberg-Coxeter subdivision 
 and discuss the convergence of a sequence of discrete surfaces defined inductively by the subdivision.
 We also study the limit set as the continuum geometric object associated with the given discrete surface.
\end{abstract}
\maketitle
\section{Introduction}
\label{sec:introduction}
One of the important problems for discrete geometry in general, 
is to find a continuum associated with a given discrete object 
and compare their geometries. 
For a triangulation of a continuous surface, 
the continuum is the continuous surface itself (for example, see \cite{Hildebrandt-Polthier-Wardetzky}). 
A typical question is 
how geometric data of the triangulation converges to 
the corresponding geometric data 
of the continuous surface when meshes get finer.
What do we do with discrete objects with no obvious underlying continuum?
To address the issue, in the present paper, we study a discrete surface,
defined in \cite{Kotani-Naito-Omori} as a trivalent graph in $\R^3$.
We introduce a method to subdivide a given discrete surface $M$, 
discuss convergence of the sequence $\{ M_{i} \}$ of the iteratively subdivided discrete surfaces,
and find a continuous object as its limit  when there is no obvious underlying surface for $M$.
\par
Let us state this more precisely. 
Let $X=(V, E, F)$ be a trivalent topological surface graph, 
where $V$ denotes the set of vertices, 
and $E$ the set of edges.
We often identify a graph $X$ with the set $V$ of its vertices. 
Although $X$ is a one-dimensional object, 
it is convenient to consider a circuit, 
a closed simple curve without self-intersections, 
as a ``face'' of $X$.
Since we assume that $X$ is a surface graph, 
the notion of faces is well-defined, 
and $F$ denotes the set of faces.
An $n$-gonal face is 
$f=\{ v_0, \ldots, v_{n-1} \}$ with the ordered vertices $v_i \in V$ in the circuit of the length $n$. 
Let us denote by $F$ the set of faces in $X$. 
Two faces are said to be neighbored when they share a common edge.
For later use, we also introduce the notion of ``leaf''.
The set  of a face $f$ and its neighboring faces is called a leaf with a core face $f$ and is denoted by $L(f)$.
\par
Given a discrete surface $\Phi \colon X \to M = \Phi(X) \subset \R^3$,
where $X$ is a trivalent graph and  $\Phi$ is a piecewise linear map, 
we let $\MV$, $\ME$, $\MF$ be the image sets of $V$, $E$, $F$.
Note that throughout the paper we write $X$ for a topological graph 
and $M$ for a graph (discrete surface) realized in $\R^3$. 
\par
Let $\{X_{i} \}$ be a sequence of the Goldberg-Coxeter construction (GC-construction, for details see Section \ref{sec:3}) iteratively constructed from $X_{0} = X$. 
For a given $M_{i} = \Phi_{i}(X_{i})$, 
its subdivision $M_{i+1}$ is constructed iteratively by the following two steps:
\begin{enumerate}
\item
  Solving the Dirichlet energy minimizing equation for $X_{i+1}$ 
  with the boundary condition $\Phi_{i}(X_i)$, 
\item
  Replace $\Phi_{i}(X_{i})$ by the barycenter of its nearest neighbors, and rename it as $M_{i}$.
\end{enumerate}
More precisely, we do the process leafwise.
We call $\{ M_i \}$ \emph{a sequence of subdivisions of a discrete surface} $M$
and prove $\{ M_i \}$ forms a Cauchy sequence in the Hausdorff topology (Theorem \ref{cauchy})
and show the energy monotonicity formula (Theorem \ref{monotonicity}).
Note that the above subdivision method is a modification of what was introduced in \cite{Tao},
and we call this procedure the Goldberg-Coxeter subdivision (GC-subdivision).
\par
The limit of this Cauchy sequence 
$\mathcal {M}_{\infty} = \overline{\bigcup M_i}$ 
is divided into three kinds of sets:
\begin{equation*}
  \mathcal {M}_{\infty} = \MM_{\MF} \cup \MM_{\MV}  \cup \MM_{\MS}.
\end{equation*} 
The first two come from accumulating points of leafwise convergence and the third one appears from global accumulation.
Given a leaf with its center $\v{f}^{(i)}$, which is an $n$-gon in $M_{i}$, 
its GC-subdivision is an $n$-gon $\v{f}^{(i+1)}$ in $M_{i+1}$ and its neighboring $n$ hexagons (see Figure~\ref{fig-b}).
The first one is the set $\MM_{\MF}$ of accumulating points associated with each face in $M_i$.
We prove in Lemma \ref{lem-d}, 
for a fixed face $\v{f}^{(i)}$ in $M_i$,  
$\{ \v{f}^{(i)}_{k}  \in M_{i+k} \} $ with $\v{f}^{(i)}_0 = \v{f}^{(i)}$ form a converging sequence
and all vertices of $\{ \v{f}^{(i)}_k \}$ 
converge to the barycenter $\v{f}^{(i)}_{\infty}$ 
of the original face $\v{f}^{(i)}$ and also of all $\v{f}^{(i)}_k$. 
We call it an accumulating point associated with the face and put 
\begin{equation*}
  \MM_\MF \coloneqq \bigcup_i  \{ \v{f}^{(i)}_{\infty} \mid \text{ the barycenter of all faces } \v{f}^{(i)} \in M_{i} \}.
\end{equation*}
\par
The second one is the set of all vertices, replaced as in the above step,  i.e., 
\begin{math}
  \mathcal{M}_{\MV} = \bigcup_{i} M_i.
\end{math}
\par
The regularity of the limit set is not trivial at all, 
although we have the energy monotonicity formula (Theorem \ref{monotonicity}).
It seems a balancing condition plays an important role.
For example, when we take a \ce{C60}, a polygonal graph on the sphere, which does not satisfy the balancing condition, 
we obtain a pathological shape as the limit of its subdivisions (Section \ref{sec:numerical}).
\par
We also prove the convergence to a point in $\MM_{\MF}$ is of $C^1$ class 
in the sense that the corresponding normal vectors 
converge to a unique unit vector independent of the choice of converging sequence in $\MM_{\MF}$.
\par
The third one is the set $\MM_{\MS}$ consisting of the rest of the accumulating points.
We know little about $\MM_{\MS}$ in general, 
however, we prove an unbranched discrete surface does not have such $\MM_{\MS}$. 
\begin{thm}[Theorem \ref{cauchy} and Theorem \ref{thm-f}]
  A sequence $\{M_i\}$ of iteratively subdivided discrete surface constructed 
  from a discrete surface $M$ forms a Cauchy sequence in the Hausdorff topology. 
  The limit set $\MM_{\infty}$ consists of $\MM_{\MV}$, $\MM_{\MF}$ and $\MM_{\MS}$.
  When $M$ is unbranched, $\MM_{\MS} $ is empty.
\end{thm}
The first statement of the above theorem was proved by the last author with a slightly different subdividing method in \cite{Tao}.
In the present paper, we propose a modified method
to further discuss
regularity/singularity in the limit of the sequence. 
We explain how we improve the original method for that purpose.
The original subdividing method defined in \cite{Tao} 
is given by solving the Dirichlet energy minimizing equation. 
It works well with a network all of whose vertices satisfy the balancing conditions but not with a network otherwise.  
This is because in the original subdividing method,
 an original vertex not satisfying the balancing condition 
and the newly generated nearest vertices are not co-planar and are forming singular points 
in the repeated process of subdivisions (see Figure~\ref{fig:planarity}).
In the modified subdividing method proposed in the present paper,
we add a smoothing step after solving the Dirichlet equation. 
Namely, we replace such a vertex by the barycenter of its nearest neighbor 
so that they are co-planar and release tensions caused by the procedure.
We also numerically calculate subdivisions of the Mackay crystal of type P
(satisfying the balancing condition) and \ce{C60} (does not satisfy it)
by both methods (see Figures~\ref{fig:planarity}, \ref{fig:c60:curvature}, and \ref{fig:mackay:curvature}).
\begin{thm}[Theorem \ref{monotonicity}]
  The total Dirichlet energy $E_D(M_i)$ is bounded when it is subjected to a finite domain at the initial stage $M_0$.
  Moreover it monotonically decreases if $M_0$ contains no $n$-gonal faces with $n > 6$ 
  and contains an $n$-gonal face with $n < 6$.
\end{thm}
\par
A trivalent graph is said to be ``branched'' when an edge is shared by more than two faces.
We note that the condition ``unbranched'' is necessary for a graph to be considered as a ``surface'' is shown in Section \ref{section:k4}. 
The $K_{4}$-lattice is the triply periodic trivalent graph in $\R^3$ discovered by T.~Sunada \cite{Sunada}, 
which is the one of the two structures that satisfy the strong-isotropic property.  
The $K_{4}$-lattice is branched. Actually each edge is shared by 10 faces. 
The numerical computation shows each leaf of the $K_{4}$-lattice converges to a smooth leaf, 
but does not converge to the same leaf (see Figure~\ref{leaf}).


\section{Preliminaries}
\label{sec:2}
There are many approaches to formulate ``Discrete Surface Theory'' based on different motivations. 
In \cite{Kotani-Naito-Omori}, a discrete surface is defined as a trivalent graph in $\R^3$ 
so that the tangent space is assigned at each vertex as the unique plane 
determined by the three nearest neighbor vertices. 
We briefly review their discussions and results.
\subsection{Discrete surface in $\R^3$ and their curvatures}
Let $X=(V,E)$ be a trivalent topological graph, 
where $V$ denotes the set of vertices, $E$ denotes the set of edges. 
The origin and the terminus of an edge $e$ are denoted by $o(e)$ and $t(e)$, respectively. 
For any $v\in V$, $E_{v}$ refers to the set of edges that emerge from $v$. 
\par
It is convenient to introduce a notion of a ``face'' although $X$ is a discrete object. 
For a circuit, a closed simple curve without self-intersections, 
we define a \emph{face} $f$ as an ordered set 
$\{ v_0, \ldots, v_{n-1} \}$ of vertices in the circuit 
and the set $F$ of \emph{faces}.  
\par
Given a trivalent topological graph $X$, 
we define a \emph{discrete surface} $M$ in $\R^3$ 
by a piecewise linear map 
$\Phi\colon X\to\R^3$ with $M = \Phi(X)$. 
Here by ``piecewise linear'', 
we mean the image of each edge $e = (v_{0}, v_{1})$ 
is given by the line segment connecting two vertices  
$\Phi (v_{0})$ and $\Phi(v_{1})$. 
\begin{definition}[Discrete Surface]
  An injective piecewise linear realization 
  $\Phi\colon X \to \R^3$ 
  of a trivalent graph $X = (V,E)$ 
  is said to be a \textnormal{discrete surface} in $\R^3$, if
  \begin{enumerate}
  \item 
    for all $v\in V$ at least two elements of $\{\Phi(e) \mid e\in E_v\}$ 
    are linearly independent in $\R^3$,
  \item 
    $\Phi(X)$ is locally oriented, that is, the order of the three edges 
    is assigned to each vertex of $X$.
  \end{enumerate}
\end{definition}
Let 
$\Phi \colon X \to M=(\MV,\ME)\subseteq \R^3$ 
be a discrete surface, in which  
$\MV = \Phi(V)$ denotes the set of vertices of $M$,
$\ME = \Phi(E)$ denotes the set of edges of $M$ 
and 
$\MF$ denotes the set of polygonal faces of $M$. 
In particular, 
we do not assume the image of a face lies on a plane or a continuous surface.
\par
As will be seen, we consider $X$ and $M$ as discrete sets 
and often identify them with the sets of vertices $V$ and $\MV$, respectively.
\par
Let 
$\v{v}=\Phi(v)$ 
and 
$\v{e}=\Phi(e)$ 
be the corresponding vertex and edge in $M$ for $v\in V$ and $e \in E$. Let 
$E_{v}=\{e_1,e_2,e_3\}$ be the oriented edges at $v$ 
and let 
$v_i$ be the tail vertex of each $e_i$. 
The \emph{tangent plane} $T_{\v{v}}M$ 
is defined as the plane with $n(\v{v})$ as its \emph{unit normal vector} $n(\v{v})$ at $\v{v}\in M$ is given by
\begin{equation*}
  n(\v{v})=
  \frac{\v{e}_{1} \times \v{e}_{2} + \v{e}_{2} \times \v{e}_{3} + \v{e}_{3} \times \v{e}_{1}}
  {|\v{e}_{1} \times \v{e}_{2} + \v{e}_{2} \times \v{e}_{3} + \v{e}_{3} \times \v{e}_{1}|}, 
  \quad (\v{e}_i = \Phi(e_i) ).
\end{equation*}
It is perpendicular to the triangle $\triangle(\v{v}_{1},\v{v}_{2},\v{v}_{3})$ with $\v{v}_i = \Phi(v_i)$.
\par
The \emph{first} and \emph{second fundamental forms} of 
$\v{v}\in M$ are given by, 
respectively
\begin{equation*} 
  \begin{aligned}
    \first(\v{v})=
    &\begin{pmatrix}
      \langle \v{e}_{2} - \v{e}_{1}, \v{e}_{2} - \v{e}_{1} \rangle, 
      &\langle \v{e}_{2} - \v{e}_{1}, \v{e}_{3} - \v{e}_{1} \rangle\\
      \langle \v{e}_{3} - \v{e}_{1}, \v{e}_{2} - \v{e}_{1} \rangle, 
      &\langle \v{e}_{3} - \v{e}_{1}, \v{e}_{3} - \v{e}_{1} \rangle
    \end{pmatrix},\\   
    \second(\v{v})=
    &\begin{pmatrix}
      - \langle \v{e}_{2} - \v{e}_{1}, \v{n}_{2} - \v{n}_{1} \rangle, 
      & - \langle \v{e}_{2} - \v{e}_{1}, \v{n}_{3} - \v{n}_{1} \rangle \\
      - \langle \v{e}_{3} - \v{e}_{1}, \v{n}_{2} - \v{n}_{1} \rangle, 
      & - \langle \v{e}_{3} - \v{e}_{1}, \v{n}_{3} - \v{n}_{1} \rangle 
    \end{pmatrix}, 
  \end{aligned}
\end{equation*}
where $\v{n}_{i}=n(\v{v}_{i})$, $i=1,\,2,\,3$. 
Note that $\second(\v{v})$ is not necessarily symmetric.
\par
\begin{definition}[Curvatures]
  Let 
  $\Phi\colon X \to M$ be a discrete surface. 
  Then for each vertex $\v{v}\in M$, 
  the 
  \textnormal{Gauss curvature $K(\v{v})$} 
  and 
  \textnormal{mean curvature $H(\v{v})$} 
  are represented as follows, respectively
  \begin{equation}
    \begin{aligned}
      &K(\v{v}) = \det[\first(\v{v})^{-1}\second(\v{v})],\\
      &H(\v{v}) = \frac{1}{2} {\tr}[\first(\v{v})^{-1}\second(\v{v})].
    \end{aligned}
  \end{equation}
\end{definition}
\begin{definition}[Discrete Minimal Surface]
  A discrete surface $\Phi\colon X \to M$ 
  is called a \textnormal{discrete minimal surface} 
  if its mean curvature vanishes at each vertex.
\end{definition}
\subsection{Discrete harmonic and minimal surfaces}
Consider a trivalent graph $X$ with weight 
$m\colon E \to \R_{+}$, satisfying 
$m(e) = m(\bar{e})$, 
where $\bar{e}$ is the reverse edge of $e$. 
\par
Let $\Phi\colon X \to M$ be a discrete surface in $\E^3$. 
For a finite subgraph $X' = (V', E') \subset X$, 
we define 
the \emph{Dirichlet energy} 
$E_{D}(\Phi |_{X'})$ as the sum of square norm of all edges, 
i.e.,
\begin{equation*}
  E_{D}(\Phi |_{X'}) = \sum_{e\in E'} m(e)|\Phi(e)|^2.
\end{equation*} 
A realization of a graph $X$ 
that minimizes the Dirichlet energy defined above for arbitrary finite subgraphs 
is called a \emph{harmonic realization} \cite{Kotani-Sunada}
or an \emph{equilibrium placement} \cite{Delgado-Friedrichs}.
\begin{prop}[Harmonic Discrete Surface (\!\!{\cite[Definition 3.15]{Kotani-Naito-Omori}})]
  A discrete surface $\Phi\colon X  \to \R^3$ 
  is harmonic with respect to the weight $m$, when it satisfies 
  \begin{equation}\label{balancing}
    m(e_{v,1})\Phi(e_{v,1}) + m(e_{v,2})\Phi(e_{v,2}) + m(e_{v,3})\Phi(e_{v,3}) = 0,
  \end{equation}
  for any $v\in V$,  and $E_{v}=\{e_{v,1},e_{v,2},e_{v,3}\}$.
\end{prop}
The equation (\ref{balancing}) is called the \emph{balancing condition}, and plays an important role later on.
\begin{prop}[\!\!{\cite[Proposition 3.16]{Kotani-Naito-Omori}}]
  \label{prop-a}
  Let 
  $F\colon X \to \R^3$ be a discrete harmonic surface with respect to the weight $m$,  
  for $v\in V$ and $E_{v}=\{e_1,e_2,e_3\}$, 
  the Gauss curvature $K(\v{v})$  
  and 
  the mean curvature $H(\v{v})$ 
  are respectively given by
  \begin{equation*}
    \begin{aligned}
      K(\v{v})
      & = -\frac{m_1 + m_2 + m_3}{2A(x)^2} 
      \sum_{i,j,k} \frac{\langle \v{e}_{i},\v{n}_{j} \rangle \langle \v{e}_{j},\v{n}_{i}\rangle}{m_j},\\ 
      H(\v{v})
      & = \frac{m_1 + m_2 + m_3}{2A(x)^2}
      \sum_{i,j,k} \frac{\langle \v{e}_{i},\v{e}_{j}\rangle(\langle \v{e}_{i},\v{n}_{j}\rangle 
        + \langle \v{e}_{j},\v{n}_{i} \rangle)}{m_j},
    \end{aligned}
  \end{equation*}
  where $m_i=m(e_i)$ and $(i,\,j,\,k)$ are the permutations of $(1,\,2,\,3)$.
\end{prop}
By Proposition \ref{prop-a}, we notice that
a discrete harmonic surface may not be minimal. 
The following theorem provides a sufficient condition 
for a discrete harmonic surface 
which has vanishing mean curvature at each vertex.
\begin{thm}[\!\!{\cite[Theorem 3.17]{Kotani-Naito-Omori}}]
  A discrete harmonic surface 
  $\Phi\colon X  \to \R^3$ is minimal 
  if for any $v\in V$ 
  and 
  $E_{v}=\{e_1,e_2,e_3\}$
  \begin{equation*}
    \langle \Phi(e_1), \Phi(e_2) \rangle 
    + \langle \Phi(e_2), \Phi(e_3)\rangle 
    + \langle \Phi(e_3), \Phi(e_1)\rangle = 0.
  \end{equation*}
  In particular, 
  if $m\colon E \to \R_{+}$ is constant, 
  the equation above is equivalent to
  \begin{equation*} 
    |\Phi(e_1)| = |\Phi(e_2)| = |\Phi(e_3)|.
  \end{equation*}
\end{thm}
\par
In the present paper, we always use $m \equiv 1$ from now on.


\section{Construction of subdivisions}
\label{sec:3}
The process of subdivision consists of two steps. 
The first step is a topological subdivision $X_i$ of $X$ by using the GC-construction, 
and the second step is to construct subdivisions $M_i$ of $M= \Phi(X)$ in $\R^3$ with $M_i= \Phi_i(X_i)$ 
so that energy of $\Phi_i$ monotonically decreases and converges to a natural continuum object.
\subsection{Goldberg-Coxeter construction of trivalent topological graphs}
The Goldberg-Coxeter construction (GC-construction) is a way 
to subdivide a trivalent surface graph 
defined by M.~Deza and M.~Detour Sikiri{\'c} \cite{Deza-Dutour} 
(see also Omori-Naito-Tate \cite{Omori-Naito-Tate}).
\begin{definition}[Goldberg-Coxeter construction]
  Let $X = (V,E,F)$ be a trivalent surface graph. 
  The graph $\GC(X)$ is built in the following steps (see Figure~\ref{fig-a}).
  \begin{enumerate}
  \item 
    \label{enum:gc:1}
    Take the dual graph $X^{*}$ of $X$. 
    Since $X$ is trivalent, $X^*$ is a triangulation, 
    namely, a surface graph whose faces are all triangles.
  \item 
    \label{enum:gc:2}
    Every triangle in $X^*$ is subdivided into another set of faces.  
    If we obtain a face which is not a triangle, 
    then it can be glued with other neighboring non-triangle faces to form triangles.
  \item 
    \label{enum:gc:3}
    By duality, 
    the triangulation of (\ref{enum:gc:2}) is transformed into $\GC(X)$.
  \end{enumerate}
\end{definition}
\begin{figure}[htp]
  \centering
  \subfigure[]{
    \label{fig-a:a} 
    \includegraphics[width=1.5in]{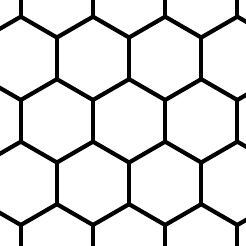}}
  \hspace{0.1in}
  \subfigure[]{
    \label{fig-a:b} 
    \includegraphics[width=1.5in]{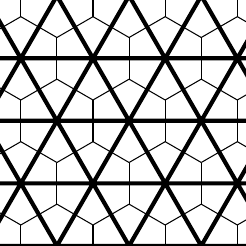}}
  \hspace{0.1in}
  \subfigure[]{
    \label{fig-a:c} 
    \includegraphics[width=1.5in]{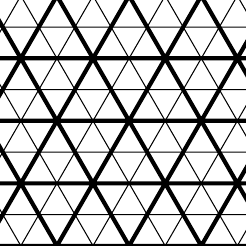}}
  \hspace{0.1in}
  \subfigure[]{
    \label{fig-a:d} 
    \includegraphics[width=1.5in]{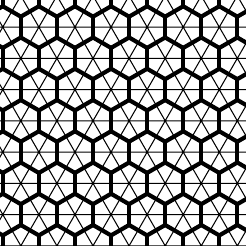}}
  \hspace{0.1in}
  \subfigure[]{
    \label{fig-a:e} 
    \includegraphics[width=1.5in]{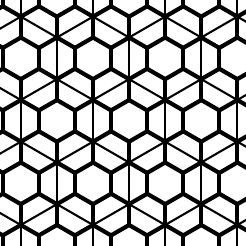}}
  \hspace{0.1in}
  \subfigure[]{
    \label{fig-a:f} 
    \includegraphics[width=1.5in]{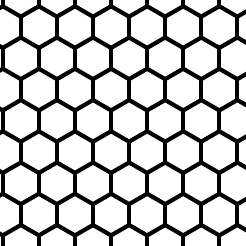}}
  \caption{the $\GC$ subdivision of the hexagonal lattice:
    For a given trivalent graph \ref{sub@fig-a:a},  
    its dual graph \ref{sub@fig-a:c} is constructed as shown in \ref{sub@fig-a:b}.
    The subdivision of \ref{sub@fig-a:c} is obtained as in \ref{sub@fig-a:d} and its dual graph, 
    \ref{sub@fig-a:f} is constructed as shown in \ref{sub@fig-a:e}.
  }
  \label{fig-a} 
\end{figure}
To apply the GC-construction for a surface graph to our case, 
we need a notion of ``leaves''. 
A \emph{leaf} with an $n$-gonal face $f$ as its core is the set
\begin{displaymath}
  L(f) = \{ f, f_1, \ldots, f_n \}
\end{displaymath}
of $f$ and all its neighboring faces $f_1, \ldots, f_n$ in $X$ (see Figure~\ref{fig-b}).
A leaf can be considered as a surface graph 
and thus be subdivided {\it topologically} by using the GC-construction (see \cite{Deza-Dutour}).
It should be noted, for a given leaf $L$ embedded in the surface,
the limit set $\overline{\cup L_i}$ of iterated subdivision $L_i$ of $L$ forms a domain in the surface in the Hausdorff topology. 
Thus we have a sequence of topological subdivisions $X_{i}$ of $X$ leafwise.
We denote the system of leafwise $\GC$-constructions by $\GC(X)$.
\begin{re}
  \mbox{}
  \begin{enumerate}
  \item
    In the present paper we only use $\GC_{2,0}$ (GC-construction of type $(2, 0)$) to subdivide the surface graph and denote it $GC$ for simplicity.
    For more general cases, see \cite{Deza-Dutour}.
  \item
    The construction of $\GC$-subdivision increases the number of hexagons of the surface graph only. 
    It does not change the number of other types of polygons.
    More precisely, on a leaf $L$
    with an $n$-gonal face $f$ at the center we obtain an $n$-gonal face $f'$ in it surrounded by $n$ hexagonal faces in $\GC(L)$ 
    (see Figure~\ref{fig-b}). 
  \item
    The limit metric on the domain is not the Euclidean metric but a similar metric studied as the tangent cone at infinity in \cite{kotani-sunada-tangentcone}.
    We do not study it in the present paper because the metric concerned in our problem is the induced metric through the realization in $\R^3$.
  \end{enumerate}
\end{re}

\begin{figure}[htp]
  \centering
  \subfigure[]{
    \label{fig-b:a} 
    \includegraphics[width=1.5in]{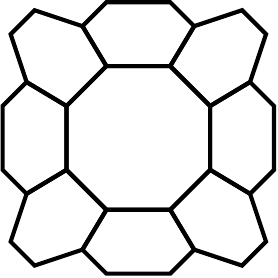}
  }
  \hspace{0.1in}
  \subfigure[]{
    \label{fig-b:b} 
    \includegraphics[width=1.5in]{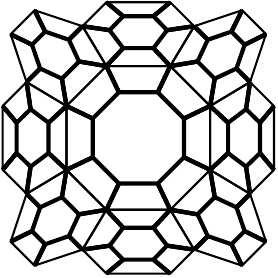}
  }
  \hspace{0.1in}
  \subfigure[]{
    \label{fig-b:c} 
    \includegraphics[width=1.5in]{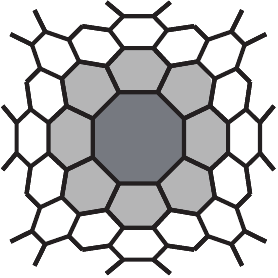}
  }
  \caption{
    Application of the GC construction to a leaf with an octagon at the center \ref{sub@fig-b:a}.
    \ref{sub@fig-b:b} shows its dual graph.
    The result as is shown at \ref{sub@fig-b:c} is a smaller octagon (gray region)
    and 8 hexagons (light-gray region) around it.
  }
  \label{fig-b}
\end{figure}
Next, we explain how to determine their configurations in $\R^3$ as a geometric subdivision of a given discrete surface.
\subsection{GC-subdivision of discrete surfaces}
For a discrete surface $\Phi \colon X \to M \subset \R^3$, 
we first introduce the method of its subdivision and 
then discuss the convergence of the sequence 
$\{ M_i \}_{i = 0}^{\infty}$  inductively constructed with $M_0 = M$ and $M_{i+1}$ as the subdivision of $M_i$.
\par
Let $X_{0} = (V_{0}, E_{0}, F_{0})$ be a trivalent topological surface graph 
and 
$X_{i+1}$ be the GC-construction of $X_{i}$, 
i.e.,
$X_{i+1}\coloneqq \GC(X_{i})$, for any $i \in \N$.
\par
Assume we have already obtained
\begin{equation*}
  \Phi_{i}\colon X_{i} \to M_{i+1},
\end{equation*}
and define 
\begin{equation*}
  \widetilde{\Phi}_{i+1} \colon X_{i+1} \to \widetilde{M}_{i+1} \subseteq \R^3
\end{equation*} 
as a minimizing map of the Dirichlet energy 
from $X_{i+1}$ with $M_i = \Phi_i (X_i)$ as the boundary condition, 
namely it satisfies
\begin{itemize}
\item[1.] 
  $\widetilde{\Phi}_{i+1}(V_i) = \Phi_{i}(V_i)$,
\item[2.]
  $\widetilde{\Phi}_{i+1}$ takes the minimum of the Dirichlet energy locally, 
  i.e.,
  on any fixed face $f^{(i)}\in F_{i}$,
  \begin{equation*}
    E_{D}(\widetilde{\Phi}_{i+1}(X_{i+1}|_{f^{(i)}})) 
    = \min \{ E_{D}(\widetilde{\Phi} \colon X_{i+1}|_{f^{(i)}} \to \E^3)\}.
  \end{equation*}
\end{itemize}
The vertex set of $\widetilde{M}_{i+1}$ is 
$\widetilde{\MV}_{i+1} = \MV_{i} \cup \MV_{i+1}^{i}$,
where $\MV_{i} = \Phi_{i}(V_{i})$, 
$\MV_{i+1}^{i}$ is the set of solution vertices of the boundary problem.
\par 
Define a projection  
\begin{equation*}
  \pi_{i+1} \colon \widetilde{M}_{i+1} \to M_{i+1}
\end{equation*}
with $M_{i+1}$ as the image. For any $\v{v} \in \widetilde{M}_{i+1}$ 
\begin{equation*}
  \pi_{i+1}(\v{v})=
  \begin{cases}
    \v{v}& \v{v} \in \MV_{i+1}^{i}\\
    \text{barycenter of its neighbors}& \v{v} \in \MV_{i}
  \end{cases}.
\end{equation*}
Finally, let 
\begin{equation*}
  \Phi_{i+1} = \pi_{i+1} \circ \widetilde{\Phi}_{i+1} \colon X_{i+1} \to M_{i+1}.
\end{equation*}
Then we define a sequence of $\{ M _{i} \}_{i}$ step by step 
as in the following diagram:
\begin{center}
  \includegraphics[bb=0 0 129 112]{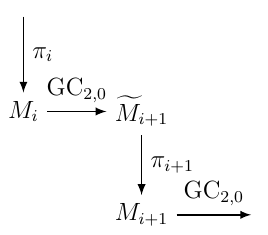}
\end{center}
It is clear there exists such $\Phi_{i+1}$ and it is unique. 
We call the Goldberg-Coxeter subdivision (GC-subdivision) $\Phi_{i+1}(X_{i+1})$.
\par
In \cite{Tao}, $\widetilde{M}_{i}$ was used as the subdivision but here we have found that the modified subdivision composing with the projection $\pi$ works better.
\subsection{Estimate of distance}
To discuss its convergence in the Hausdorff topology, 
we use the following energy estimate on a face:
\begin{prop}[\!\cite{Tao}]
  \label{prop-m}
  For any fixed $n$-gonal face $\v{f}^{(i)}\in F_{i}$, 
  there exists a constant number $\lambda(n)<1$ 
  such that
  \begin{equation}\label{energy-1}
    E_{D}(\widetilde{\Phi}_{i+1}(X_{i+1}|_\v{f^{(i)}}))
    \le \lambda(n) E_D(\Phi_{i}(X_{i}|_{f^{(i)}})).
  \end{equation}
\end{prop}
\par
We make a quick review of the proof in our setting for the reader's convenience.
\begin{proof}
  It is sufficient to prove the assertion for the case $i=0$.
  \par
  For fixed $\v{f}^{(0)}$, 
  the vertices of $\Phi_{0}(X_{0}|_{f^{(0)}})$ are denoted by 
  $\{\v{v}_{0_1},\, \v{v}_{0_2},\ldots,\v{v}_{0_{n-1}}\}$, 
  and the inner vertices of $\widetilde{\Phi}_{1}(X_{1}|_{f^{(0)}})\in M_{1}$ are denoted by 
  $\{\v{v}_{1_1},\, \v{v}_{1_2},\ldots,\v{v}_{1_{n-1}}\}$,
  where 
  $\v{v}_{0_i}$ 
  and
  $\v{v}_{1_i}$
  are connected by a single edge. 
  \par
  Let 
  $\v{f}^{(0)} = (\v{v}_{0_1},\, \v{v}_{0_2},\ldots, \v{v}_{0_{n-1}})^t$, 
  $\v{f}^{(1)} = (\v{v}_{1_1},\, \v{v}_{1_2},\ldots, \v{v}_{1_{n-1}})^t$,
  \begin{equation*}
    T \coloneqq 
    \begin{pmatrix}
      O &I_{n-1}\\
      1& O^T
    \end{pmatrix} 
    \in \mathcal{M}(n),
  \end{equation*}
  where $O = (0, 0, \ldots,0)^{t} \in \mathbb{R}^{n-1}$.
  Then
  \begin{equation*}
    \begin{aligned}
      E_{D}(\widetilde{\Phi}_{1}(X_{1}|_{f^{(0)}})) 
      = \sum_{i=0}^{n-1} |\v{v}_{0_i} - \v{v}_{1_i}|^2 
      + \sum_{i=0}^{n-1} |\v{v}_{1_i} - \v{v}_{1_{i+1}}|^2 
      = \| \v{f}^{(0)} - \v{f}^{(1)}\|^2 + \| \v{f}^{(1)} - T \v{f}^{(1)} \|^2,
    \end{aligned}
  \end{equation*}
  \begin{equation*}
    E_D(\Phi_{0}(X_{0}|_{f^{(0)}}))
    = \sum_{i=0}^{n-1} |\v{v}_{0_i} - \v{v}_{0_{i+1}} |^2
    = \| \v{f}^{(0)} - T\v{f}^{(0)}\|^2,
  \end{equation*}
  where $|\cdot|$ is the vector norm and $\| \cdot \|$ is the Hilbert-Schmidt norm of the square matrix.
  \par
  The minimizer of Dirichlet energy infers that 
  \begin{equation*}
    \frac{\partial E_{D}(\widetilde{\Phi}_{1}(X_{1}|_{f^{(0)}}))}{\partial \v{v}_{1_i}} = 0,
    \quad
    \textnormal{ for } i = 0,\,1, \ldots, n-1,
  \end{equation*}
  which is 
  \begin{equation*}
    - \v{v}_{1_{i-1}} + 3\v{v}_{1_i} - \v{v}_{1_{i+1}} 
    = \v{v}_{0_i},
    \quad
    \textnormal{ for } i = 0,\,1, \ldots, n-1,
  \end{equation*}
  Then we obtain 
  \begin{equation}
    \label{3.2}
    \v{f}^{(1)} = A(n) \v{f}^{(0)},
  \end{equation} 
  where
  \begin{equation}
    A(n)\coloneqq
    (3I_{n}-T-T^{t})^{-1} 
    =
    \begin{pmatrix}
      3&-1&0&\ldots&0&-1\\
      -1&3&-1&\ldots&0&0\\
      0&-1&3&\ldots&0&0\\
      0&0&-1&\ldots&0&0\\
      &&&\ldots&&\\
      0&0&0&\ldots&3&-1\\
      -1&0&0&\ldots&-1&3
    \end{pmatrix}^{-1}
    \in \mathcal{M}(n),
  \end{equation}
  and $I_{n}$ is the identity matrix of size $n$.
  \par
  Direct computation shows 
  the eigenvalues of $A(n)$ to be
  \begin{equation}
    \lambda_{k}(n) = \frac{1}{1+4\sin^2(k\pi/n)},
    \quad k=0,1,\ldots,n-1.
  \end{equation}
  On the other hand, 
  since $A(n)$ is symmetric, we have
  \begin{equation*}
    A(n)(\v{f}^{(0)} - T \v{f}^{(0)}) 
    = A(n)\v{f}^{(0)} - A(n)T\v{f}^{(0)}
    = A(n)\v{f}^{(0)} -TA(n)\v{f}^{(0)}
    = \v{f}^{(1)} - T\v{f}^{(1)}.
  \end{equation*} 
  Here we claim that
  \begin{equation}
    \v{f}^{(0)} - T \v{f}^{(0)}\perp \phi_0,
    \quad
    \v{f}^{(1)} - T\v{f}^{(1)}\perp \phi_0,
    \label{eq-a}
  \end{equation}
  where $\phi_{0}=(1,\ldots,1)^T$ 
  is the eigenvector of $\lambda_{0}=1$.\\
  In fact, if we let 
  $\v{f}^{(0)} = c^0\phi_0 + \v{f}^{(0)}_{\perp}$, 
  where $\v{f}^{(0)}_{\perp} \perp \phi_0$,  
  then by noticing that $T\phi_{0} = \phi_{0}$, 
  we have  
  \begin{equation*}
    \begin{aligned}
      \langle \v{f}^{(0)} - T\v{f}^{(0)}, \phi_0\rangle
      & = \langle \v{f}^{(0)}_{\perp} - T\v{f}^{(0)}_{\perp}, \phi_0\rangle
      = \langle \v{f}^{(0)}_{\perp}, \phi_0\rangle 
      -\langle T\v{f}^{(0)}_{\perp}, \phi_0 \rangle\\
      & = -\langle T\v{f}^{(0)}_{\perp}, T\phi_0\rangle
      = - T\langle \v{f}^{(0)}_{\perp}, \phi_0\rangle \\
      & = 0.
    \end{aligned}
  \end{equation*}
  Similarly, 
  we can also prove 
  $\v{f}^{(1)} - T\v{f}^{(1)}\perp \phi_0$ 
  by noticing the fact that 
  \begin{equation*}
    A(n)\phi_0 = \lambda_0(n)\phi_0 = \phi_0.
  \end{equation*}
  \par
  Letting $\tilde{\sigma}(A(n))$ be the second largest eigenvalue of $A(n)$,
  by (\ref{eq-a}) we have
  \begin{equation}
    \label{eq:localenergy0}
    \|\v{f}^{(0)} - T\v{f}^{(1)}\|^2 
    \le \tilde{\sigma}(A(n))^2 \| \v{f}^{(0)} - T\v{f}^{(0)} \|^2
    = \lambda_1^2(n)\| \v{f}^{(0)}-T\v{f}^{(0)} \|^2.
  \end{equation}
  Similarly, 
  \begin{equation*}
    \| \v{f}^{(0)} - \v{f}^{(1)} \|^2
    \le\lambda_{1}(n)(1-\lambda_{1}(n))\|\v{f}^{(0)} - T\v{f}^{(0)}\|^2.
  \end{equation*}
  Therefore
  \begin{equation}
    \label{eq:localenergy1}
    \begin{aligned}
      &\| \v{f}^{(1)} - T\v{f}^{(1)}\|^2 + \| \v{f}^{(0)} - \v{f}^{(1)}\|^2
      \\
      & \le \lambda_1^2(n)\| \v{f}^{(0)} - T\v{f}^{(0)}\|^2 
      + \lambda_{1}(n)(1-\lambda_{1}(n))\| \v{f}^{(0)} - T\v{f}^{(0)}\|^2\\
      & =\lambda_1(n)\| \v{f}^{(0)} - T\v{f}^{(0)}\|^2,
    \end{aligned}
  \end{equation}
  where $\lambda_{1}(n) = 1/(1+4\sin^2(\pi/n))<1$ as desired.
\end{proof}


\section{Convergence of subdivided discrete surfaces}
\label{sec:4}
\subsection{Cauchy sequence}
Firstly we prove the sequence of subdivided discrete surfaces forms a Cauchy sequence in the Hausdorff topology. 
\begin{thm}\label{cauchy}
  The sequence of discrete surfaces 
  $\{M_{i}\}^{\infty}_{i=0}$ 
  that are constructed by the GC-subdivisions  as in Section \ref{sec:3} forms a Cauchy sequence 
  in the Hausdorff topology.
  \label{thm-e}
\end{thm}
\begin{proof} 
  Let $f^{(i)}$ be a fixed $n$-gonal face in $X_{i}$, 
  and $f^{(i+1)}$ be the face defined by the inner vertices of 
  $X_{i+1}|_{f^{(i)}}$. 
  Consider the Hausdorff distance 
  \begin{equation}
    \begin{aligned}
      d_H(\Phi_{i}(\partial f^{(i)}), \widetilde{\Phi}_{i+1}(\partial f^{(i+1)}))
      &\le\sum_{e\in X_{i+1}|_{f^{(i)}}} | \widetilde{\Phi}_{i+1}(e)|\\
      &\le\sqrt{2nE_{D}(\widetilde{\Phi}_{i+1}(X_{i+1}|_{f^{(i)}}))}
      \le E_{0}(n)\sqrt{\lambda_{1}^{i+1}(n)},
    \end{aligned}
  \end{equation}
  where 
  $E_{0}(n) \coloneqq \sqrt{2nE_{D}(\Phi_{0}(\partial f^{(0)}))}$ 
  is constant and is determined by 
  $\Phi_{0}\colon X_{0} \to M_{0}$ 
  and 
  $\lambda_{1}(n)=1/(1+4\sin^2(\pi/n))$.
  Since each face of a fixed 3-valent graph has finitely many edges, 
  the number of vertices of each face ($n$ of $n$-gonal face) in $X_{0}$ is bounded from above.
  Letting
  $\lambda_{1} = \max\{ \lambda_{1}(n) \}$, 
  $E = \max\{ E_{0}(n) \}$,
  we have
  \begin{equation}
    \begin{aligned}
      d_{H}(M_{i}, \widetilde{M}_{i+1}) = 
      \sup_{\v{f}^{(i)} \in \MF_{i}} \{d_H(\Phi_{i}(\partial f^{(i)}), \widetilde{\Phi}_{i+1}(\partial f^{(i+1)}))\} 
      \le E\sqrt{\lambda_{1}^{i+1}}.
    \end{aligned}
  \end{equation}
  On the other hand, taking $\v{v} \in \v{f}^{(i)}$, also we have $\v{v} \in \MV_{i}$. 
  Since $\pi_{i+1}(\v{v})$ is the barycenter of its nearest neighbors, it is easy to see
  \begin{equation}
    \begin{aligned}
      d_{H}(\v{v}, \pi_{i+1}(\v{v})) < \sup_{\v{f}^{(i)} \in \MF_{i}} \{d_H(\Phi_{i}(\partial f^{(i)}), \widetilde{\Phi}_{i+1}(\partial f^{(i+1)}))\}
      \le E\sqrt{\lambda_{1}^{i+1}}.
    \end{aligned}
  \end{equation}
  That is,
  \begin{equation}
    \begin{aligned}
      d_{H}(M_{i}, \pi_{i+1}(\widetilde{M}_{i+1})) = d_{H}(M_{i}, M_{i+1}) 
      = \sup \{d_{H}(\v{v}, \pi_{i+1}(\v{v})) \}
      \le E\sqrt{\lambda_{1}^{i+1}}.
    \end{aligned}
  \end{equation}
  Thus 
  for any $\epsilon>0$,  
  let $N=\lceil2\log_{1/\lambda_{1}}(\Lambda/\epsilon)\rceil$. 
  Then for any $i$, $j>N$ $(j>i)$, 
  we have 
  \begin{equation*}
    \begin{aligned}
      d_{H}(M_{i}, M_{j})
      &\le d_{H}(M_{i}, M_{i+1}) 
      + d_{H}(M_{i+1}, M_{i+2}) 
      + \cdots + d_{H}(M_{j-1}, M_{j})\\
      &\le E\Big(\sqrt{\lambda^{i+1}_{1}} + \sqrt{\lambda^{i+2}_{1}} 
      + \cdots + \sqrt{\lambda^{j}_{1}}\Big)\\
      &<\Lambda E \sqrt{\lambda^{i+1}_{1}}\\
      &<\epsilon,
    \end{aligned}
  \end{equation*}
  where 
  $\Lambda = (1+\sqrt{\lambda_{1}})/(1-\lambda_{1})$ 
  is a constant determined by 
  $\Phi_{0}\colon X_{0}\to M_{0}$ as well. 
\end{proof}
\subsection{Monotonicity of the Dirichlet energy}
Let $\{ M_{i} \}_{i}^{\infty}$ be the sequence of discrete surfaces, 
and $\Phi_{i} \colon X_{i} \to M_{i} = (\MV_{i}, \ME_{i}, \MF_{i})$ be a discrete surface at the $i$-th step constructed from an unbranched bounded domain in a discrete surface.
In this subsection, we show the monotonicity of the Dirichlet energy.
It is sufficient to prove it on the  energy of the subdivision sequence constructed from a leaf in $M=M_0$.
The core idea is simple.
When we take a subdivision,
the size of each face gets smaller and smaller, but the number of faces increases.
Fortunately, however, the number of $n$-gonal faces with $n \neq 6$ does not change in the subdivision process, 
but the number of hexagonal faces only increases.
We study the sum of the energy of hexagonal faces to find it well balanced.
\par
For any fixed $i$, $\MF_{i}$ consists of two parts as
\begin{equation}
  \begin{aligned}
    \MF_{i} = \MF_{i}^{n<6} \cup \MF_{i}^{=6} \cup \MF_{i}^{n>6},
  \end{aligned}
\end{equation}
where $\MF_{i}^{n=6}$ is the set of hexagonal faces in  $M_i$  and
$\MF_{i}^{n<6}$ and $\MF_{i}^{n>6}$ are the sets of  $n$-gonal faces in  $M_i$ with $ n < 6$ and $n >6$, respectively.  
Note firstly there is the largest $n$, which we denote $N$, because we are working with a bounded domain, and secondly  we have
\begin{equation}
  \begin{aligned}
    \sharp \MF_{i}^{n<6} = \sharp \MF_{0}^{n<6}, \quad  \sharp \MF_{i}^{n>6} = \sharp \MF_{0}^{n>6}, 
  \end{aligned}
\end{equation}
since the GC-subdivision increases the number of hexagonal faces only.
\par
Let
\begin{displaymath}
  E_{D}(\v{f}) = \sum_{\v{v} \sim \v{w} \in \v{f}} | \v{v}- \v{w} |^2.
\end{displaymath}
\begin{lem}
  \label{lem-e}
  Let $\v{f}$ be an $n$-gonal face in $M_i$, 
  and $\v{f}'$ be a face in $M_{i+1}$ as a solution of the Dirichlet problem with the boundary $\v{f}$, 
  and $\tilde{\v{f}}$ 
  be the solution of the Dirichlet problem with the boundary $\v{f}$ and edges connecting the corresponding vertices (see Figure~\ref{fig:face}).
  Then we obtain 
  \begin{align}
    \label{energy-of-solution0}
    &E_{D}(\v{f}') \le \lambda_1^{2}(n) E_{D}(\v{f}), \\
    \label{energy-of-solution1}
    &E_{D}(\tilde{\v{f}}) \le \lambda_1(n) E_{D}(\v{f}),
  \end{align}
  where $\lambda_1^{2}(n)$ is the spectrum radius of $A(n)$ computed in the previous section.
\end{lem}
\begin{proof}
  The inequalities (\ref{energy-of-solution0}) 
  and (\ref{energy-of-solution1}) are equivalent to (\ref{eq:localenergy0}) and (\ref{eq:localenergy1}), respectively.
\end{proof}
\begin{figure}[ht]
  \centering
  \includegraphics[scale=1.00]{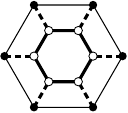}
  \caption{
    Assume $\v{f} \in \MF_i$ consists of thin edges, 
    then $\v{f}' \in \MF_{i+1}$ consists of thick edges, 
    and $\tilde{\v{f}}$ consists of thick dashed edges.
    Here black vertices are vertices in $\MV_i$ and white vertices are in $\MV_{i+1}$.
  }
  \label{fig:face}
\end{figure}
\begin{thm}[monotonicity of  the Dirichlet energy]\label{monotonicity}
  Let $\{ M_{i} \}_{i}^{\infty}$ be the sequence of discrete surfaces constructed from a leaf or a bounded domain.
  The Dirichlet energy of $M_{i}$ is bounded from above by a constant independent of $n$.
  Moreover it monotonically decreases if $M_0$ contains no $n$-gonal faces with $n > 6$ 
  and contains an $n$-gonal face with $n < 6$.
\end{thm}
\begin{proof}
  Since each edge is shared by two faces, we obtain 
  \begin{equation*}
    E_{D}(M_{i}) = \sum_{\v{e}\in \ME_{i}} |\v{e}|^2 = \frac{1}{2} \sum_{\v{f}^{(i)} \in \MF_{i}} E_{D}(\v{f^{(i)}}).
  \end{equation*}
  For any $\v{f}^{(i)} \in \MF_{i}$, 
  let $\widetilde{\v{f}}^{(i)}$ be the set of vertices of $\v{f}^i$ and $\v{f}^{(i+1)} \in M_{i+1}$, the solution of the Dirichlet problem with the boundary $\v{f}^{(i)}$,  and the edges connecting the corresponding vertices as in Lemma \ref{lem-e}.
  Then, 
  \begin{equation*}
    E_{D}(M_{i+1}) \le E_D(\widetilde{M}_{i+1}) = \sum_{\v{f}^{(i)}\in \MF_{i}} E_{D}(\widetilde{\v{f}}^{(i)}).
  \end{equation*}
  By using Lemma \ref{lem-e}, 
  now we compute the Dirichlet energy of $M_{i+1}$ as 
  \begin{equation}
    \label{eq:4.8}
    \begin{aligned}
      E_D(M_{i+1})
      &\le
      E_D(\widetilde{M}_{i+1}) 
      =
      \sum_{\v{f}\in \MF_i} E_D(\tilde{\v{f}}) 
      \le
      \sum_{\v{f}\in \MF_i} \lambda_1(n) E_D(\v{f}) 
      \\
      &=
      \sum_{\v{f}\in \MF_i^{n<6}} \lambda_1(n) E_D(\v{f}) 
      +
      \frac{1}{2} \sum_{\v{f}\in \MF_i^{n=6}} E_D(\v{f}) 
      +
      \sum_{\v{f}\in \MF_i^{n>6}} \lambda_1(n) E_D(\v{f}) 
      \\
      &=
      \frac{1}{2}
      \left( 
        \sum_{\v{f}\in \MF_i^{n<6}} E_D(\v{f}) 
        +
        \sum_{\v{f}\in \MF_i^{n=6}} E_D(\v{f}) 
        +
        \sum_{\v{f}\in \MF_i^{n>6}} E_D(\v{f}) 
      \right)
      \\
      &+
      \sum_{\v{f} \in \MF_i^{n<6}} \left(\lambda_1(n) - \frac{1}{2}\right) E_D(\v{f})
      +
      \sum_{\v{f} \in \MF_i^{n>6}} \left(\lambda_1(n) - \frac{1}{2}\right) E_D(\v{f})
      \\
      &=
      E_D(M_i)
      +
      \sum_{\v{f} \in \MF_i^{n<6}} \left(\lambda_1(n) - \frac{1}{2}\right) E_D(\v{f})
      +
      \sum_{\v{f} \in \MF_i^{n>6}} \left(\lambda_1(n) - \frac{1}{2}\right) E_D(\v{f}).
    \end{aligned}
  \end{equation}
  We also have
  \begin{equation}
    \label{eq:4.9}
    \begin{alignedat}{3}
      \lambda_1(n) &< 1/2 &\quad &\text{ for } n<6, \\
      \lambda_1(6) &= 1/2, \\
      \lambda_1(n) &< 1 &\quad &\text{ for } 6<n.
    \end{alignedat}
  \end{equation}
  If there are no $n$-gonal faces ($n > 6$) and at least one $n$-gonal face ($n < 6$), 
  then, by (\ref{eq:4.8}), we obtain 
  \begin{displaymath}
    E_D(M_{i+1})
    \le 
    E_D(M_i)
    +
    \sum_{\MF_i^{n<6}} \left(\lambda_1(n) - \frac{1}{2}\right) E_D(\v{f}).
  \end{displaymath}
  Since $\lambda_1(n) < 1/2$ ($n < 6$), we obtain 
  \begin{equation}
    \label{eq:monotone_decreasing}
    E_D(M_{i+1}) <  E_D(M_i).
  \end{equation}
  Inequality (\ref{eq:monotone_decreasing}) implies that
  the Dirichlet energy of $M_i$ is monotonically decreasing.
  \par
  On the other hand, by (\ref{eq:4.8}), we also obtain 
  \begin{equation}
    \label{eq:generalcase}
    \begin{aligned}
      E_D(M_{i+1})
      \le
      E_D(M_i)
      &+
      \sum_{\v{f} \in \MF_0^{n<6}} 
      \left(\lambda_1(n) - \frac{1}{2}\right) \lambda_1(n)^{2i} E_D(\v{f})
      \\
      &+
      \sum_{\v{f} \in \MF_0^{n<6}} 
      \left(\lambda_1(n) - \frac{1}{2}\right) \lambda_1(n)^{2i} E_D(\v{f}).
    \end{aligned}
  \end{equation}
  Hence, we obtain 
  \begin{equation}
    \label{eq:7}
    E_D(M_{i+1}) \le E_D(M_i)
    + 
    \sum_{n=3}^5 
    C_n \lambda_1(n)^{2i}
    +
    \sum_{n=7}^N
    C_n \lambda_1(n)^{2i}, 
  \end{equation}
  where $C_n = (\lambda_1(n) - 1/2) N_n E_n$, $N_n$ is the number of $n$-gonal faces in $M_0$, 
  and $E_n = \max\{E_D(\v{f}) : \v{f} \in M_0, \, \v{f} \text{ is an $n$-gonal face}\}$.
  Note that, for $n \not= 6$, $C_n$ are independent of $i$.
  Finally, we obtain 
  \begin{equation}
    \label{eq:8}
    E_D(M_i)
    \le
    E_D(M_0)
    + \sum_{n=3}^5 C_n \frac{1 - \lambda_1(n)^{2i-2}}{1-\lambda_1(n)^2}
    + \sum_{n=7}^N C_n \frac{1 - \lambda_1(n)^{2i-2}}{1-\lambda_1(n)^2}, 
  \end{equation}
  and (\ref{eq:8}) implies the boundedness of $E_D(M_i)$.
\end{proof}


\section{The limit set $\MM_{\infty}$}
Let
$M_{0} = \{\MV_{0}, \ME_{0}, \MF_{0} \}$ be a $3$-valent graph in $\R^{3}$ and $\{ M_{i} = \{\MV_{i}, \ME_{i}, \MF_{i}\} \}_{i=0}^{\infty}$
be the sequence constructed by the GC-subdivision.
The limit  set in the Hausdorff topology is divided into three kinds:
\begin{displaymath}
  \MM_{\infty} = \MM_{\MF} \cup \MM_{\MV} \cup \MM_{\MS}.
\end{displaymath}
The first two come from accumulating points of the leafwise convergence and the third one emerges as a global accumulation.
\subsection{unbranched surfaces}
\label{sec:5.3}
For a general discrete surface $M$,
we know little about $\MM_{\MS}$ in general,
but under a natural condition, we prove $\MM_\MS$ is empty.
\par
When every edge of $M$ is shared by two faces only, we say $M$ is unbranched.
\begin{thm}
  \label{thm-f}
  Let
  $M_{0} = \{\MV_{0}, \ME_{0}, \MF_{0} \}$
  be a $3$-valent graph in $\E^{3}$ which satisfies
  \begin{enumerate}
  \item
    Each edge of $M_{0}$ is shared by at most two faces.
  \item
    Any two faces intersect at one edge or not at all.
  \item
    The convex hulls $\con(L(\v{f}_1))$ of leaves $L(\v{f}_1)$ and the convex hull $\con(L(\v{f}_2))$ of $L(\v{f}_2)$ 
  intersect when either $\v{f}_1 \sim \v{f}_2$ 
  {\upshape(}$\v{f}_1$ and $\v{f}_2$ share a common edge{\upshape)}
  or there is a connecting face $\v{f}_1 \diamond \v{f}_2$ of $\v{f}_1$ and $\v{f}_2$.
  \end{enumerate}
  Then
  \begin{math}
    \MM_{\infty} = \MM_{\MV} \cup \MM_{\MF}.
  \end{math}
\end{thm}
In the following, we prove Theorem \ref{thm-f}.
A leaf with an $n$-gonal face $\v{f}$ as its core is a set
\begin{displaymath}
  L(\v{f}) =\{ \v{f}, \v{f}_1, \cdots, \v{f}_n\}
\end{displaymath}
of $\v{f}$ and the neighboring faces $\v{f}_{\alpha}$, $\alpha = 1, \ldots, n$ of $\v{f}$.
The set of vertices of faces belonging to $L(\v{f})$ is denoted by $\MV(L(\v{f}))$.
\begin{lem}
  \begin{displaymath}
    \bigcup_{\v{F} \in \MF(M_{i+1})} \con (L(\v{F})) \subset 
    \bigcup_{\v{f} \in \MF(M_i)} \con (L(\v{f})), 
  \end{displaymath}
  where $\con(\Omega)$ is the convex hull of the set $\Omega$.
\end{lem}
\begin{proof} 
In the subdividing process, we have two kinds of faces; the first kind is
obtained as a solution $f'$ of the Dirichlet problem with the boundary
condition $\v{f}$ by the equation (\ref{3.2}). 
Let us denote $\v{f}' = \mathcal{A}\v{f}$. 
Note that
\begin{displaymath}
  \MV(\v{f}') = A \MV (\v{f}),
\end{displaymath}
and the leaf with $\v{f}'$ as its core is
\begin{displaymath}
  L(\v{f}') = \{ \v{f}',\, \v{f}' \diamond \v{f}_1', \cdots, \v{f}' \diamond \v{f}_n' \}.
\end{displaymath}
\par
The second kind is a face connecting two faces $\v{f}' = \mathcal{A}(\v{f})$ and
$\v{f}_{\alpha}' = \mathcal{A}(\v{f}_{\alpha})$ ($\alpha \in \{1, \cdots n\}$) of
the first kind. We denote it $\v{f}' \diamond \v{f}_{\alpha}'$.  
The set of its vertices is
\begin{displaymath}
  \MV(\v{f}'\diamond \v{f}_{\alpha}') \subset
  \MV(\v{f}') \cup \MV(\v{f}_{\alpha}') \cup \MV(\v{f})
  = A\MV(\v{f}) \cup A\MV(\v{f}_{\alpha})\cup \Pi \MV(\v{f}),
\end{displaymath}
and the leaf with $\v{f}'\diamond \v{f}_{\alpha}'$ as its core is
\begin{displaymath}
  L(\v{f}'\diamond \v{f}_{\alpha}') = \{  \v{f}'\diamond \v{f}_{\alpha}', \v{f}', \v{f}_{\alpha}',
  \v{f}'\diamond \v{f}_{\alpha-1}',  \v{f}'\diamond \v{f}_{\alpha+1}', \v{f}_{\alpha}' \diamond
  \v{f}_{\alpha-1}',  \v{f}_{\alpha}' \diamond \v{f}_{\alpha+1}' \}.
\end{displaymath}
\par
Let $\MF(M_i)$ be the set of all faces in $M_i$ and notice that for a given $\v{F}\in \MF(M_{i+1})$, 
there is a face $\v{f} \in \MF(M_i)$ such that $\v{F} \in L(\v{f})$.
More precisely $\v{F}$ is either a solution face $\v{f}' = \mathcal{A}\v{f}$ or a
connecting face $\v{f}' \diamond \v{f}_{\alpha}'$.
\begin{figure}[ht]
  \centering
  \subfigure[]{\label{fig:leaffigue:a}\includegraphics[scale=0.80]{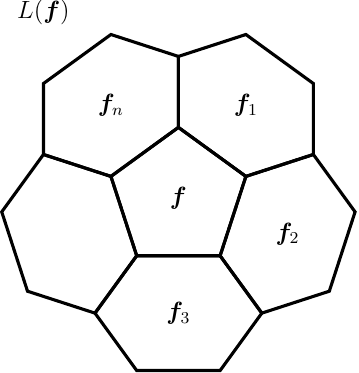}}
  \subfigure[]{\label{fig:leaffigue:b}\includegraphics[scale=0.80]{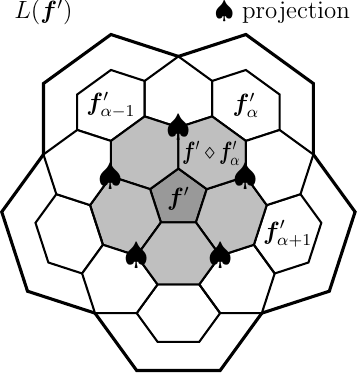}}
  \subfigure[]{\label{fig:leaffigue:c}\includegraphics[scale=0.80]{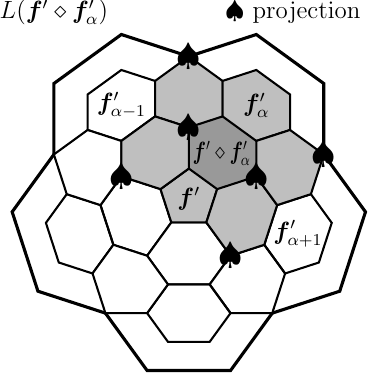}}
  \caption{
    \ref{sub@fig:leaffigue:a}
    $L(\v{f}) = \{\v{f}, \v{f}_1, \ldots, \v{f}_n\}$, 
    \ref{sub@fig:leaffigue:b}
    $L(\v{f}')$ consists of the gray face and light-gray faces,
    \ref{sub@fig:leaffigue:c}
    $L(\v{f}' \diamond \v{f}'_{\alpha})$ consists of the gray face and light-gray faces.
    }
  \label{fig:leaffigure}
\end{figure}
\par
For the first case, namely
for $\v{F} = \v{f}' = \mathcal{A}\v{f}$ in $\MF(M_{i+1})$ with $\v{f} \in \MF(M_i)$,
we have
\begin{math}
  \con (L(\v{F})) \subset \con(L(\v{f})).
\end{math}
Since we have the relation
\begin{displaymath}
  \MV (L(\v{f}))  \subset A \MV(\v{f})  \cup \left(\cup_{\alpha} A\MV(\v{f}_{\alpha})\right) \cup \pi\MV(\v{f}),  
\end{displaymath}
where $\pi$ is the action of taking the barycenter of the three nearest neighboring vertices.
They are all combination of elements of $\MV(L(\v{f}))$, 
and this relation yields the claim.
\par
For the second case, namely for
$\v{F} = \v{f}'\diamond \v{f}_{\alpha}'$ in $\MF(M_{i+1})$ with $\v{f} \in \MF(M_i)$ and a neighboring face $\v{f}_{\alpha}$ of $\v{f}$,
\begin{displaymath}
  \con (L(\v{F})) \subset \con( L(\v{f}) ) \cup \con (L(\v{f}_{\alpha})).
\end{displaymath}
We also have the relation
\begin{displaymath}
  \MV (L(\v{F})) \subset \pi \MV(\v{f}) \cup  \pi \MV(\v{f}_{\alpha}) \cup A \MV(\v{f}) \cup A\MV(\v{f}_{\alpha}) \cup A\MV(\v{f}_{\alpha-1}) \cup A \MV(\v{f}_{\alpha +1}),  
\end{displaymath}
where $\pi$ is the action of taking the barycenter of the three nearest neighboring vertices, 
and here we use the vertices of $L(\v{f}_{\alpha})$ only.
Therefore elements in $\MV (L(\v{F}))$ are again all combination of elements of $\MV(L(\v{f}))$ and $\MV(L(\v{f}_\alpha))$.
\par
Putting those two cases together, we have
\begin{displaymath}
  \bigcup_{\v{F} \in \MF(M_{i+1})} \con (L(\v{F})) \subset 
  \bigcup_{\v{f} \in \MF(M_i)} \con (L(\v{f})). 
\end{displaymath}
\end{proof}
We define
\begin{displaymath}
  \MC_{i}: = \bigcup_{\v{f} \in \MF(M_i)} \con (L(\v{f})).  
\end{displaymath}
The lemma gives
\begin{displaymath}
  \MM_{\infty} \subset \cdots \subset \MC_{i+1} \subset \MC_{i} \subset \cdots \subset \MC_{0}.  
\end{displaymath}
\par
For any $\v{x}_{\infty} \in \MM_{\infty}$, assume that $\v{x}_{\infty} \notin \MM_{\MV}$ and take a sequence of vertices $\v{x}_k$ such that $\lim \v{x}_k = \v{x}_{\infty}$.
Because $\v{x}_{\infty} \notin \MM_{\MV}$, we can assume no two $\v{x}_k$ and $\v{x}_j$ are in the same stage,
i.e., there is a unique $\v{x}_i \in \MF(M_i)$ for every $i$ without loss of generality.
\par
Let $\v{x}_i \in \v{f}_i$ and $\v{x}_{i+1} \in \v{f}_{i+1}$, then we have
\begin{displaymath}
  \v{f}_{i+1} = \mathcal{A} \v{f}_i
\end{displaymath}
or 
\begin{displaymath}
  \v{f}_{i+1} = \mathcal{A}\v{f}_i \diamond \mathcal{A} \v{f}_{i,\alpha}, \quad \v{f} \sim \v{f}_{i, \alpha} \in \MF(M_{i}), 
\end{displaymath}
since we assume the two convex hulls of leaves $L(\v{f})$ and $L(\v{f}_{\alpha})$ 
intersect only when $\v{f} \sim \v{f}_{\alpha}$ or there is a connecting face $\v{f} \diamond \v{f}_{\alpha}$.
\par
In the latter case, $\v{x}_{i+1} \in \v{f}_i \cap \v{f}_{i,\alpha}$, which contradicts the choice of the sequence. 
Therefore
\begin{displaymath}
  \v{x}_{i+1} \in \mathcal{A}\v{f}_i, \quad \v{x}_i \in \v{f}_i.
\end{displaymath}
That implies $\v{x}_{\infty}$ is an accumulation point of a face, 
that is, $\v{x}_{\infty} \in \MM_{\MF}$, 
completing the proof of Theorem \ref{thm-f}.
\subsection{The limit sets associated with faces.}
The $\MM_{\MF}$ is the set of accumulating points associated with each face in $M_i$. We have the following
\begin{prop}
  \begin{equation*}
    \MM_\MF \coloneqq \bigcup_i  \{ \v{f}^{(i)}_{\infty} \mid \text{ the barycenter of all faces } \v{f}^{(i)} \in M_{i} \}.
  \end{equation*}
\end{prop}
Recall for a leaf with its center $\v{f}^{(i)}$, which is an $n$-gon in $M_{i}$, 
its GC-subdivision is an $n$-gon $\v{f}^{(i+1)}$ in $M_{i+1}$ and its neighboring $n$ hexagons (see Figure~\ref{fig-b}). 
\begin{lem}
  \label{lem-d}
  For any $\v{f}^{(i)} \in \MF_{i}$, 
  let $\v{f}^{(i+1)} = A\v{f}^{(i)}$.
  Then in the sense of vertices of a face,
  \begin{displaymath}
    \v{f}^{(i+1)} \subset \con (\v{f}^{(i)}).
  \end{displaymath}
\end{lem}
\begin{proof}
  Let $A = (a_{lm})$ ($l,\,m=0,\,1,\ldots,n-1$), 
  $\v{f}^{(i)} = (\v{v}_{i_0}, \v{v}_{i_1}, \ldots, \v{v}_{i_{n-1}})^t$. 
  \par
  Noticing that $A \cdot 1 = 1$,   
  then for any $l$, 
  \begin{equation}
    \sum_{m=0}^{n-1}a_{lm}\v{v}_{i_{m}} = \v{v}_{(i+1)_{l}},
  \end{equation}
  where 
  \begin{equation}
    \sum_{m=0}^{n-1} a_{lm} = 1,\quad a_{lm} \geq 0.
  \end{equation}
  That is, 
  $$\v{f}^{(i+1)} \subset \con (\v{f}^{(i)}).$$ 
\end{proof}
\par
It also shows $\v{f}^{(i)}$ and $\v{f}^{(i+1)}$ share the same barycenter $\v{f}^{b}$.
Furthermore, 
since $\v{f}^{(i+k)} = A^k \v{f}^{(i)}$, 
by Proposition \ref{prop-m}, 
\begin{equation}
  E_{D}(\v{f}^{(i+k)}) < \lambda^{k}_{1} E_{D}( \v{f}^{(i)} ) \to 0
  \text{ as }k\to \infty, 
\end{equation}
which means $\v{f}^{(i)}$ degenerates to a single point as $i$ goes to $\infty$. 
We call this point $\v{f}^{\infty}$ the accumulation point associated with $\v{f}^{(i)}$.
It is easy to see for any $k$, 
$\v{f}^{\infty}$ and $\v{f}^{b}$ are lying in the convex hull of $\v{f}^{(i+k)}$.
Therefore $\v{f}^{\infty} = \v{f}^{b}$.
\subsection{The limit sets associated with vertices}
$\MM_{\MV}$ is the set of all vertices,  i.e., 
\begin{equation}
  \mathcal{M}_{\MV} = \bigcup_{i} M_i.
\end{equation}
The convergence to a point in $\MM_{\MV}$ is pathological, 
although we have the energy monotonicity formula (Theorem \ref{monotonicity}).
It seems a balancing condition plays an important role.
\par
For example,
when we take the atomic configuration of the fullerene \ce{C60},
a polygonal graph on the sphere, which does not satisfy the balancing condition,
we obtain a pathological shape as the limit of its subdivisions.
 It seems the modified method gives a better convergence than the original method proposed earlier \cite{Tao}.
For numerical calculations for \ce{C60}, Mackay crystal of type P and their subdivisions, see Section \ref{sec:numerical}.


\section{Examples: \ce{C60} and Mackay crystals of type P}
\label{sec:numerical}
We include the numerical tests on both \ce{C60} and the Mackay crystal 
(Figure~\ref{fig:C60-Mackay}).
Figures~\ref{fig:planarity}-\ref{fig:mackay:curvature}
and Tables~\ref{table:c60:curvature}-\ref{table:mackay:curvature} 
are at the end of this paper.
\begin{figure}[H]
  \centering
  \subfigure[]{\label{fig:C60-Mackay:C60}\includegraphics[width=80pt]{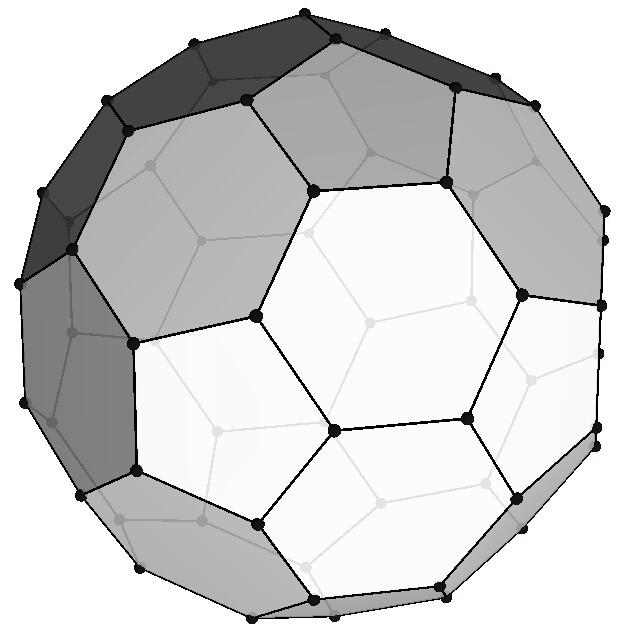}}
  \hspace{0.1in}
  \subfigure[]{\label{fig:C60-Mackay:Mackay}\includegraphics[width=80pt]{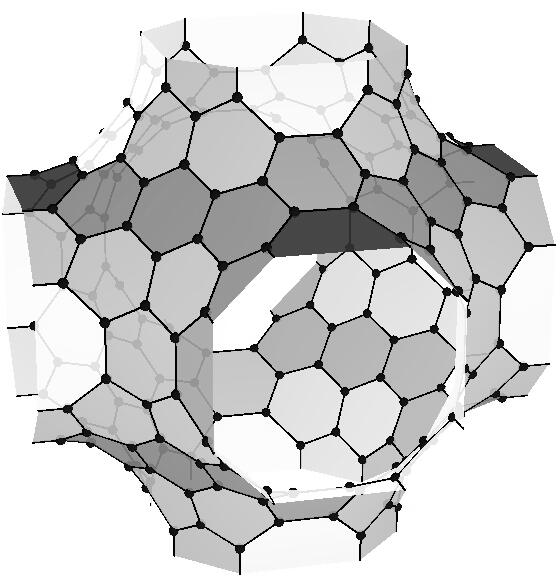}}
  \caption{
    \ref{sub@fig:C60-Mackay:C60}
    \ce{C60}, 
    \ref{sub@fig:C60-Mackay:Mackay}
    Mackay crystal of type P.
  }
  \label{fig:C60-Mackay}
\end{figure}
\par
The Mackay crystal is a carbon network introduced by Mackay and Terrones \cite{Mackay-Terrones:1991} as  a discrete triply periodic minimal surface (Schwarz P surface). 
In \cite{Kotani-Naito-Omori}, its geometry is carefully studied. 
\par
\ce{C60} is the atomic structure of the famous fullerene and is studied as a carbon network on the sphere.
Because each carbon atom has three bonds, we can apply our method to study its subdivisions.
\ce{C60} is a good test to illustrate difficulties in Discrete Surface Theory as it has ``positive curvature'' and does not satisfy the balancing condition. 
\par
Firstly we point out the Mackay crystal satisfies the balancing condition (as it is a discrete minimal surface) while \ce{C60} is not.
In the present paper, we modify the original subdivision method we used in \cite{Tao}.
Numerical computations (Figure \ref{fig:planarity}) show we have smooth convergence to the Mackay crystal in both the original and modified subdivisions, 
while a better situation appears in the modified rather than the original method with  \ce{C60}.
We observe singularities appearing at the vertices of the given discrete surface and that of the subdivided discrete surfaces and therefore modify the original subdivision method so that vertices at each step satisfy the balancing conditions (see the first row in Figure~\ref{fig:planarity}).
\par
Now we exhibit numerical tests of curvatures.
The Gauss curvature and mean curvature of the sequence of subdivisions constructed from the Mackay crystal 
(Figure \ref{fig:mackay:curvature} and Table \ref{table:mackay:curvature})
are computed and showed their convergence.
The Gauss curvature remains negative and the mean curvature goes to zero.
In particular, the limit surface exactly is the Schwarz P surface (minimal surface).
The Gauss curvature and mean curvature of the sequence of subdivisions constructed from the \ce{C60} 
(Figure \ref{fig:c60:curvature} and Table \ref{table:c60:curvature}) are also computed.
The convergence is better in the modified version.
Curvature seems to concentrate at the barycenter of the pentagons in the modified version and at all vertices in the original version. 
\par
Lastly, we include the graph of the Dirichlet energy (Figure \ref{fig:energy}, cf. Theorem \ref{monotonicity}).
\ce{C60} has hexagonal faces and pentagonal faces.
The energy monotonically decreases since it has no face with $n$-gons ($n>6$).
On the other hand, the Mackay has hexagonal faces and octagonal faces.
The energy monotonically increases since it has no face with $n$-gons ($n<6$) but we have an upper bound for the energy. 
It would be interesting to study the regularity of the convergence.
\begin{figure}[H]
  \centering
  \begin{tabular}{cc}
    \includegraphics[width=200pt]{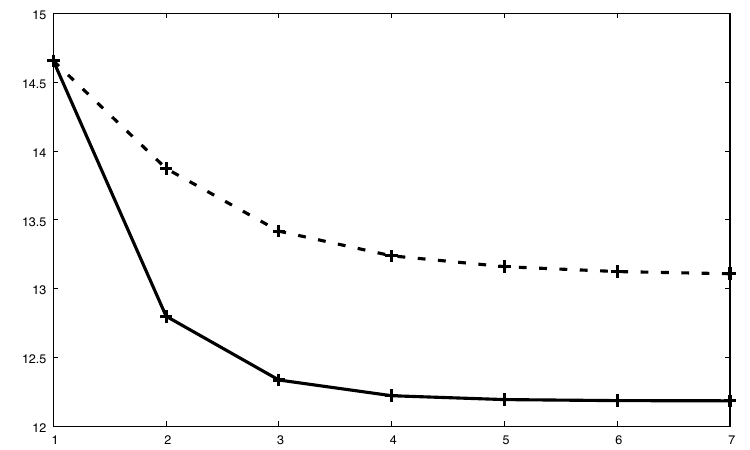}
    &\includegraphics[width=200pt]{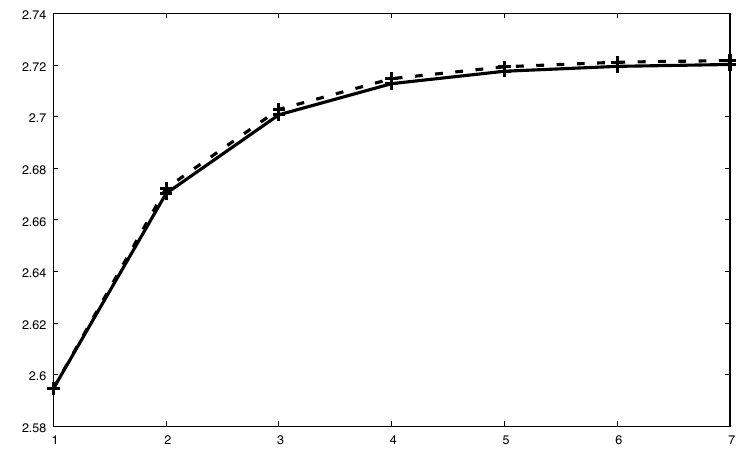}\\
    \ce{C60}&Mackay crystal
  \end{tabular}
  \caption{
    The Dirichlet energies for subdivision of \ce{C60} and Mackay crystal.
    Dashed line: original subdivisions.
    Solid line: modified subdivisions.
  }
  \label{fig:energy}
\end{figure}


\section{Branched surface: the $K_4$-lattice}
\label{section:k4}
In this section, we study an example of a branched surface whose limit surface is branched.
When a face $f$ has a branched edge, 
i.e., an edge which is shared with more than two faces, 
the Goldberg-Coxeter construction cannot be done for the whole graph 
but only for the leaf with the central face $f$. 
For each leaf, we take the subdivision process 
and obtain its limit surface as we proved in the previous section. 
\par
Now we see such an example.
The $K_{4}$-lattices is a triply periodic trivalent graph in $\R^3$ discovered by Sunada \cite{Sunada}, 
as one of the two structures satisfy the strong-isotropic property.  
The $K_{4}$-lattice is branched, actually, as each edge is shared by 10 faces. 
The numerical computation shows each leaf of $K_{4}$ converges to a smooth leaf.
Two leaves, however, can have all common neighboring faces 
but still not converge to the same leaf. 
\begin{figure}[ht]
  \centering
  \subfigure[]{\label{fig:leaf:K4}\includegraphics[width=80pt]{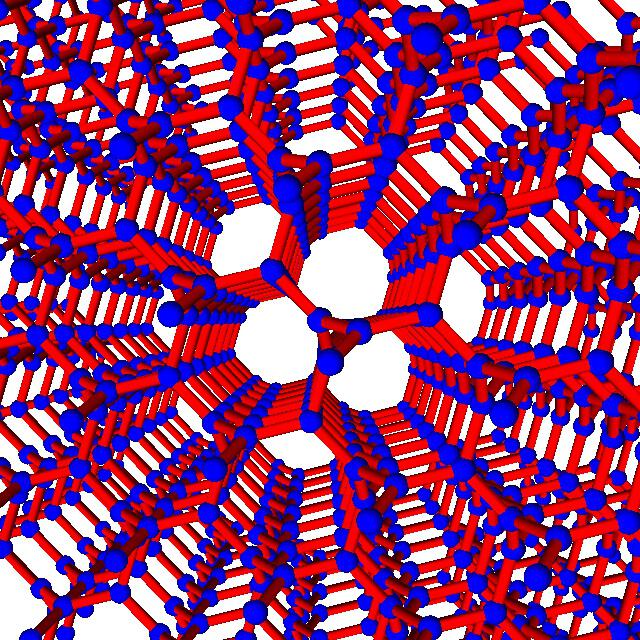}}
  \hspace{0.1in}
  \subfigure[]{\label{fig:leaf:leaf01}\includegraphics[width=80pt]{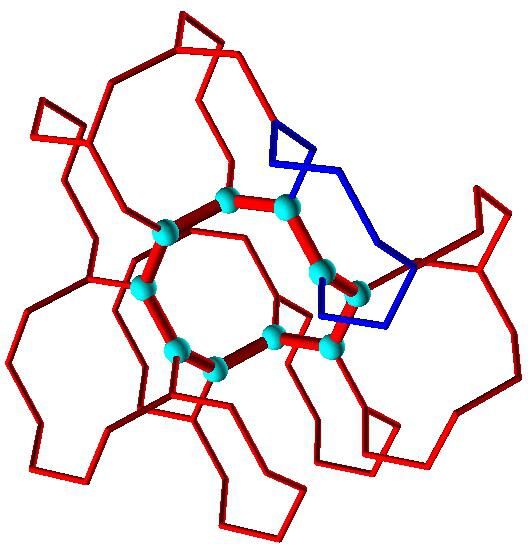}}
  \hspace{0.1in}
  \subfigure[]{\label{fig:leaf:leaf02}\includegraphics[width=80pt]{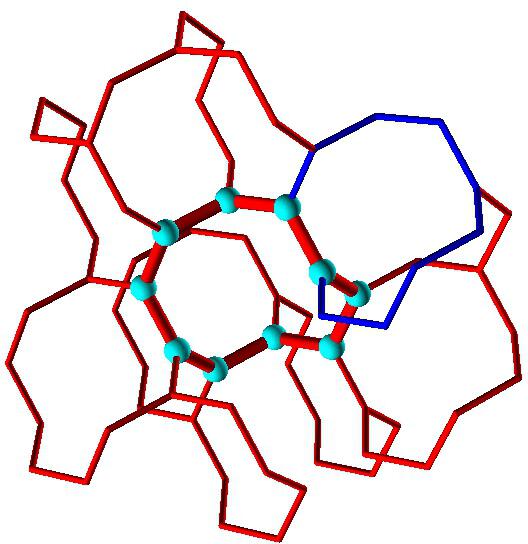}}
  \hspace{0.1in}
  \subfigure[]{\label{fig:leaf:leaf11}\includegraphics[width=80pt]{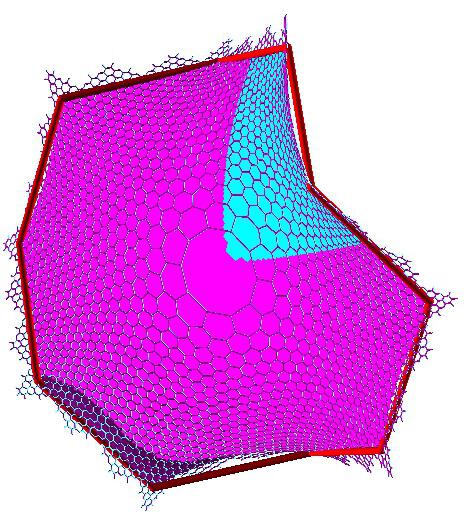}}
  \hspace{0.1in}
  \subfigure[]{\label{fig:leaf:leaf12}\includegraphics[width=80pt]{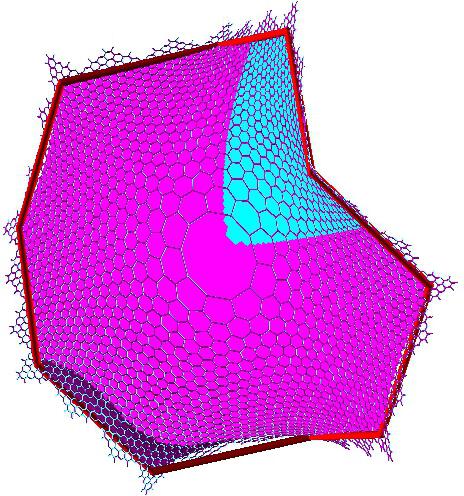}}
  \caption{
    \ref{sub@fig:leaf:K4} K4 lattice: triply periodic lattice.
    \ref{sub@fig:leaf:leaf01} and \ref{sub@fig:leaf:leaf02} Leaves with a central face $\v{f}_0$ (drawn by bold lines), 
    $\v{f}_0$  has 10 different leaves and each leaf has $\v{f}_0$ as its center and
    ten neighboring faces. \ref{sub@fig:leaf:leaf01} and \ref{sub@fig:leaf:leaf02} have 9 common faces ($\v{f}_2, \ldots, \v{f}_9$)
    (drawn by red thin lines) and a different face (drawn by blue thin lines)
    ($\v{f}_{1a}$ in \ref{sub@fig:leaf:leaf01}, $\v{f}_{1b}$ in \ref{sub@fig:leaf:leaf02}), respectively.
    \ref{sub@fig:leaf:leaf11} limit set constructed from the leaf 
    $L_a = \{\v{f}_0,\v{f}_{1a},\v{f}_2,\ldots,\v{f}_9\}$.
    \ref{sub@fig:leaf:leaf12} limit set constructed from the leaf $L_b =\{\v{f}_0, \v{f}_{1b}, \v{f}_2,\ldots, \v{f}_9\}$, 
    \ref{sub@fig:leaf:leaf11} and \ref{sub@fig:leaf:leaf12}
    are the results of five-time subdivisions and the mesh is 
    showing three-time subdivisions. The blue areas are different parts of two limit sets.
  }
  \label{leaf}
\end{figure}



\begin{figure}[H]
  \centering
  \begin{tabular}{l|cccc}
    &3&4&5&6\\
    \hline
    \raisebox{35pt}{\parbox[c]{50pt}{original\\\ce{C60}}}
    &\includegraphics[width=80pt]{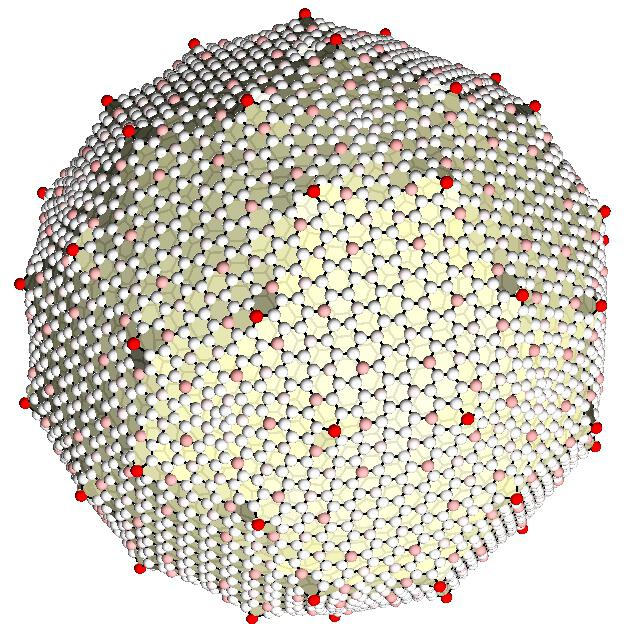}
      &\includegraphics[width=80pt]{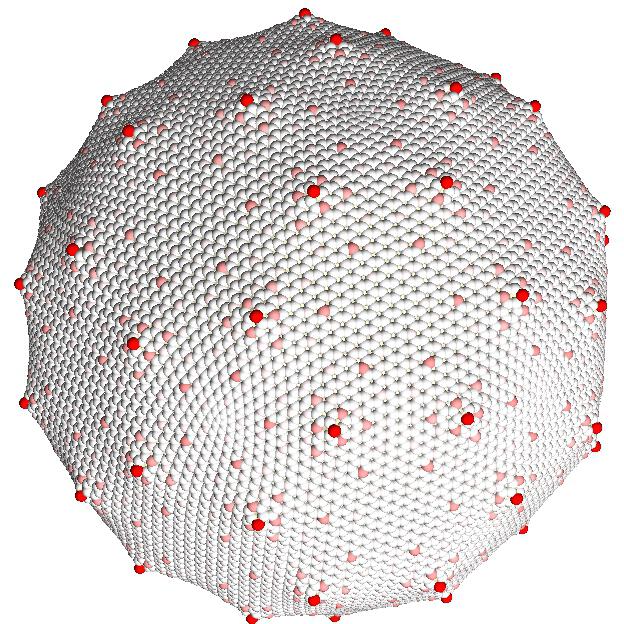}
        &\includegraphics[width=80pt]{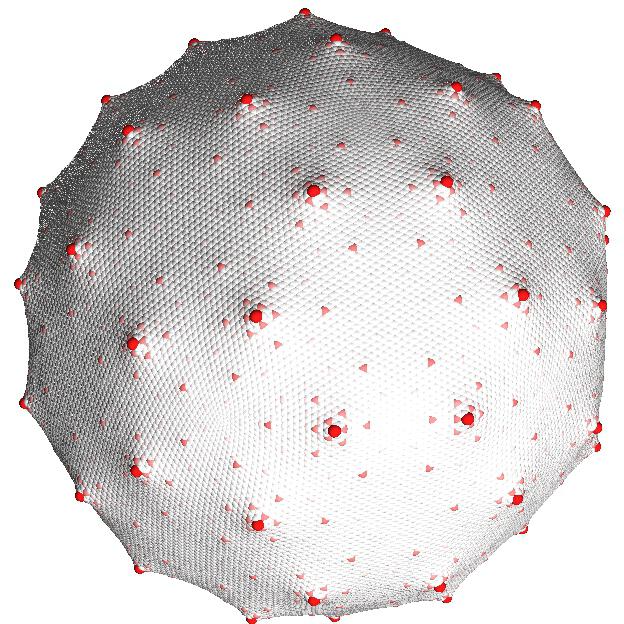}
          &\includegraphics[width=80pt]{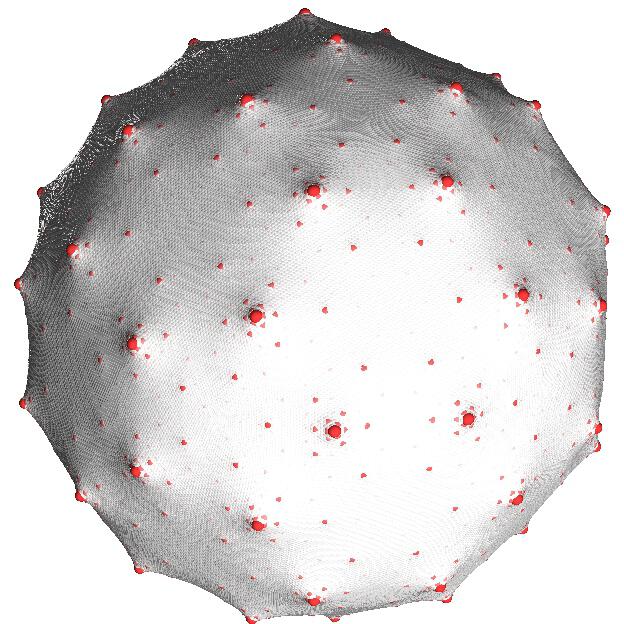}\\
    \hline
    \raisebox{35pt}{\parbox[c]{50pt}{modified\\\ce{C60}}}
    &\includegraphics[width=80pt]{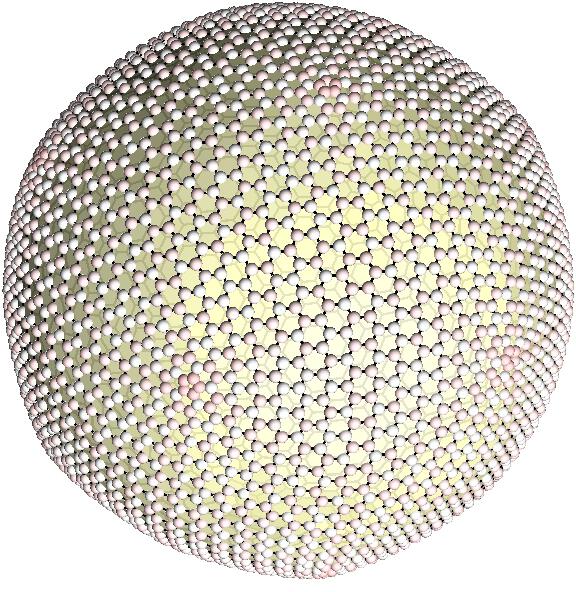}
      &\includegraphics[width=80pt]{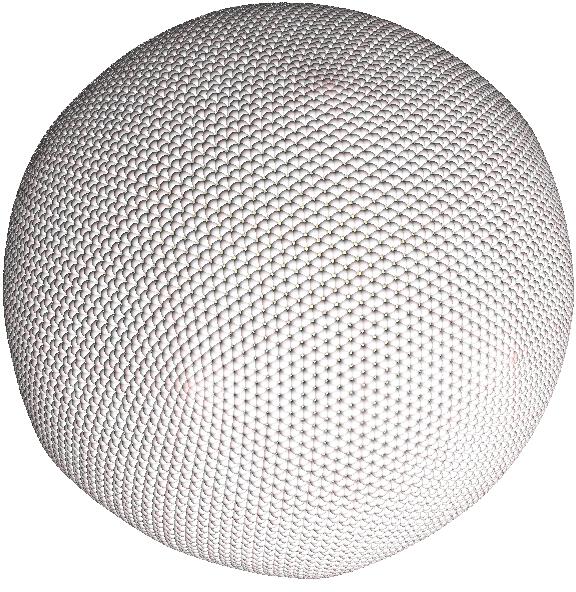}
        &\includegraphics[width=80pt]{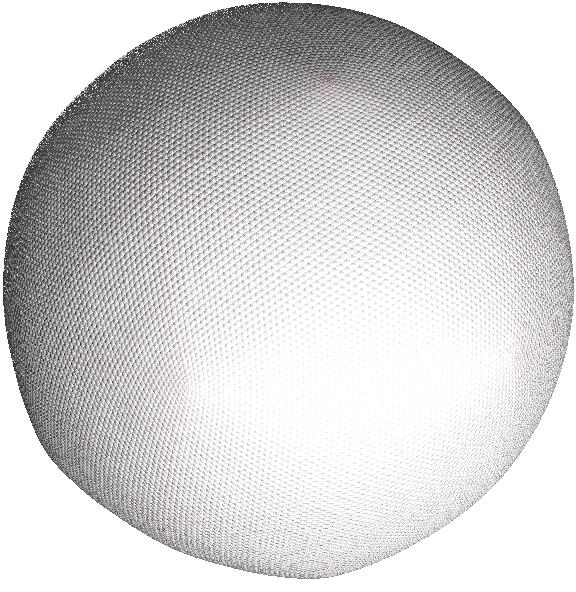}
          &\includegraphics[width=80pt]{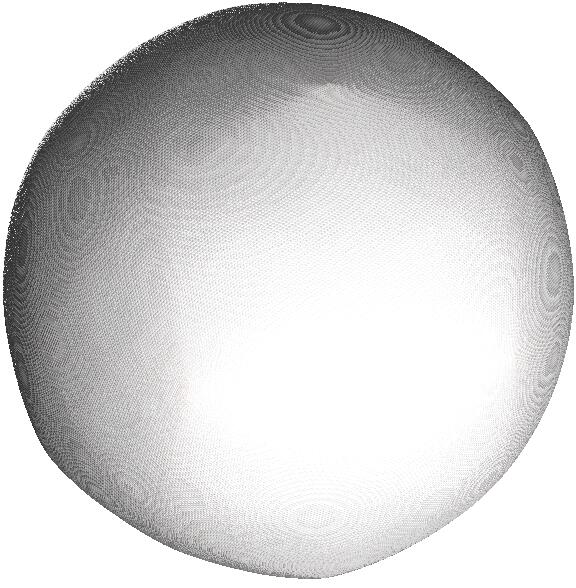}\\
    \hline
    \hline
    \raisebox{35pt}{\parbox[c]{50pt}{original\\ Mackay}}
    &\includegraphics[width=80pt]{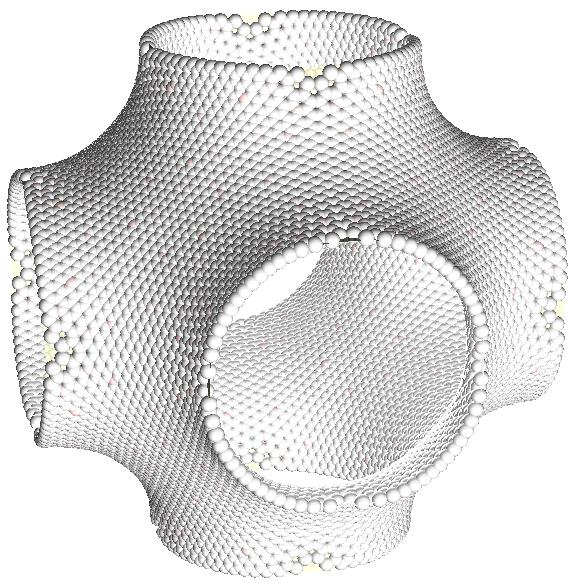}
      &\includegraphics[width=80pt]{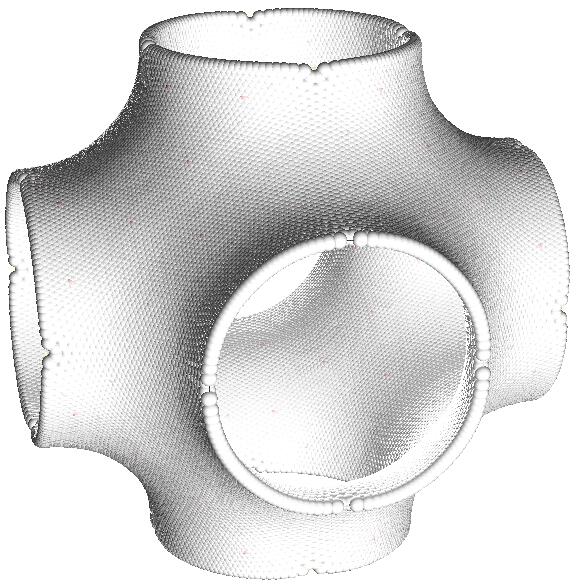}
        &\includegraphics[width=80pt]{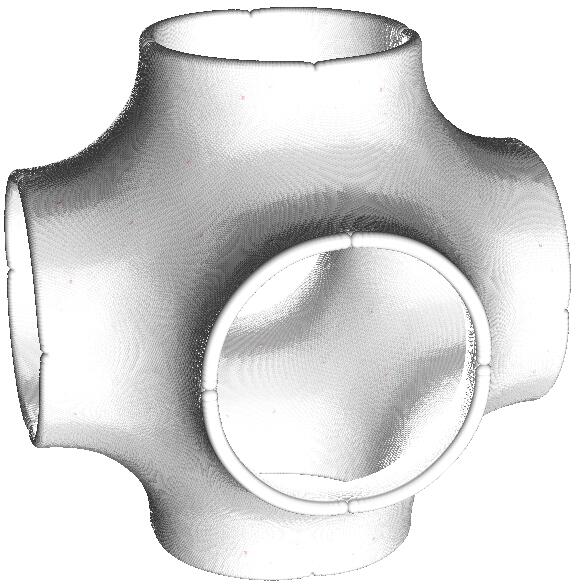}
          &\includegraphics[width=80pt]{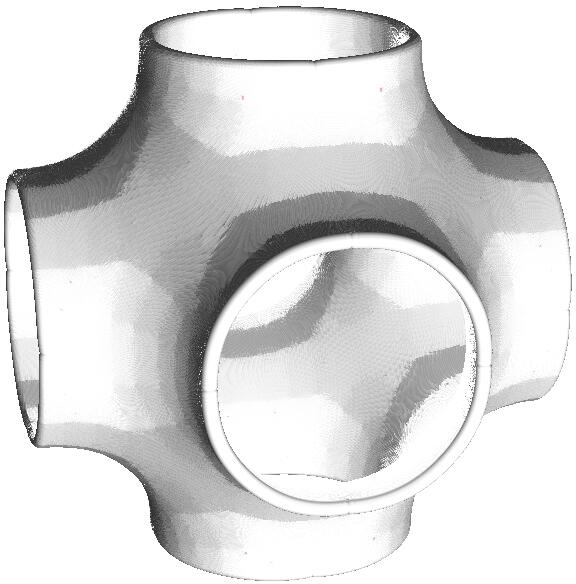}\\
    \hline
    \raisebox{35pt}{\parbox[c]{50pt}{modified\\ Mackay}}
    &\includegraphics[width=80pt]{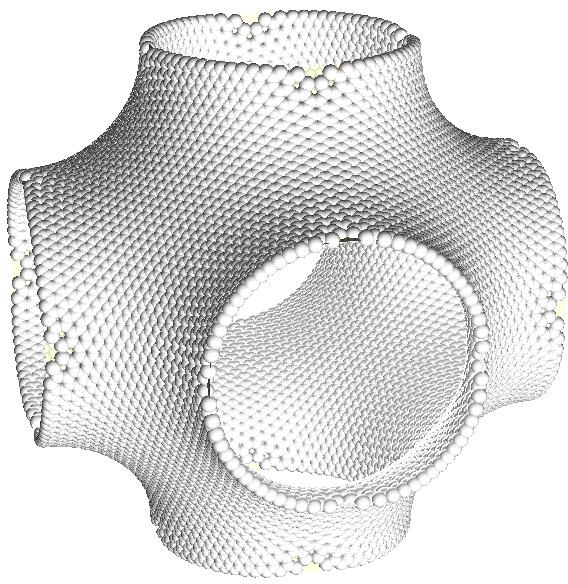}
      &\includegraphics[width=80pt]{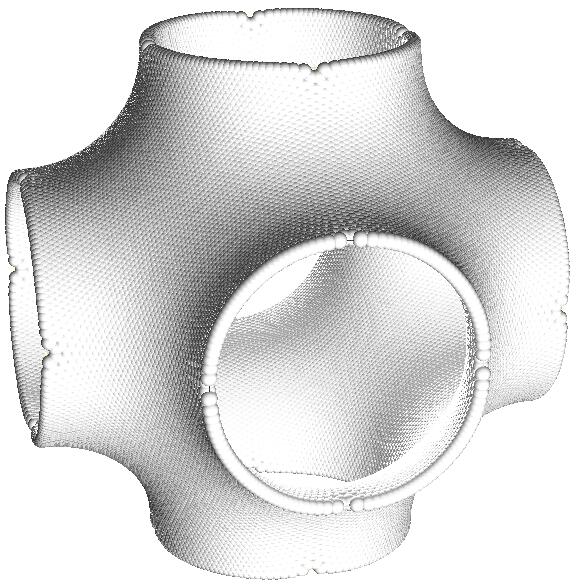}
        &\includegraphics[width=80pt]{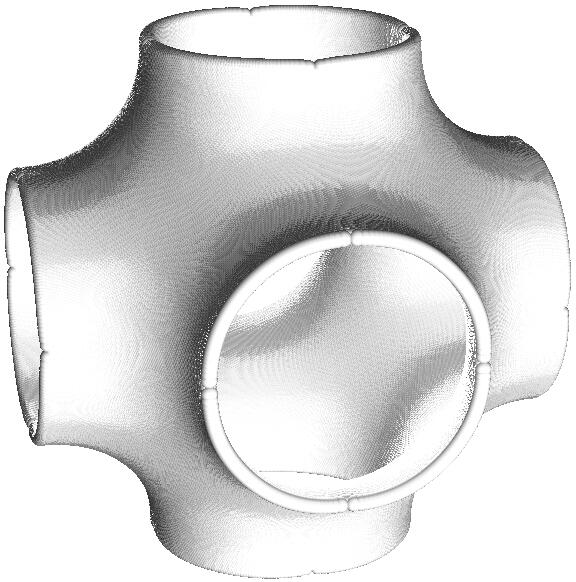}
          &\includegraphics[width=80pt]{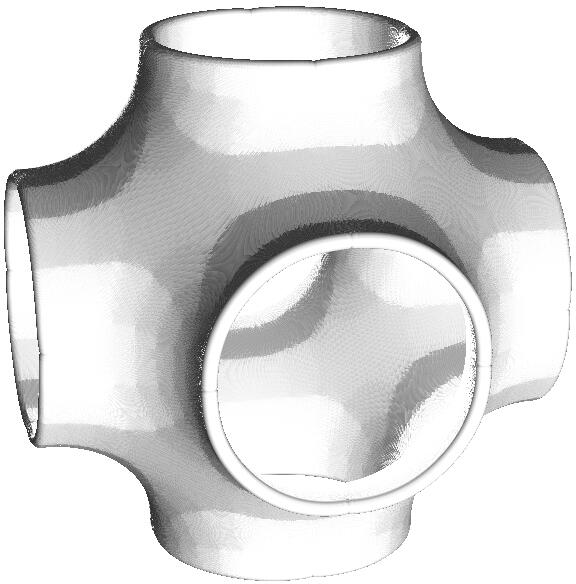}\\
  \end{tabular}
  \caption{
    Numerical computations of subdivisions (from three times up through six times) of \ce{C60} and Mackay crystal.
    Red points in original subdivisions of \ce{C60} are angled, 
    which are located at original vertices of it (see Figure~\ref{fig:C60-Mackay:C60}).
  }
  \label{fig:planarity}
\end{figure}
\begin{figure}[H]
  \centering
  \begin{tabular}{l|cccc}
    &3&4&5&6\\
    \hline
    \raisebox{35pt}{\parbox[c]{50pt}{original\\Gauss curvature}}
    &\includegraphics[width=80pt]{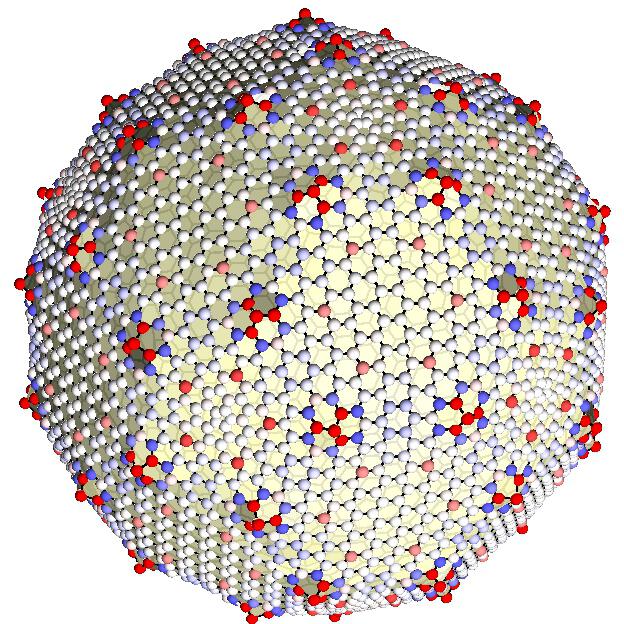}
      &\includegraphics[width=80pt]{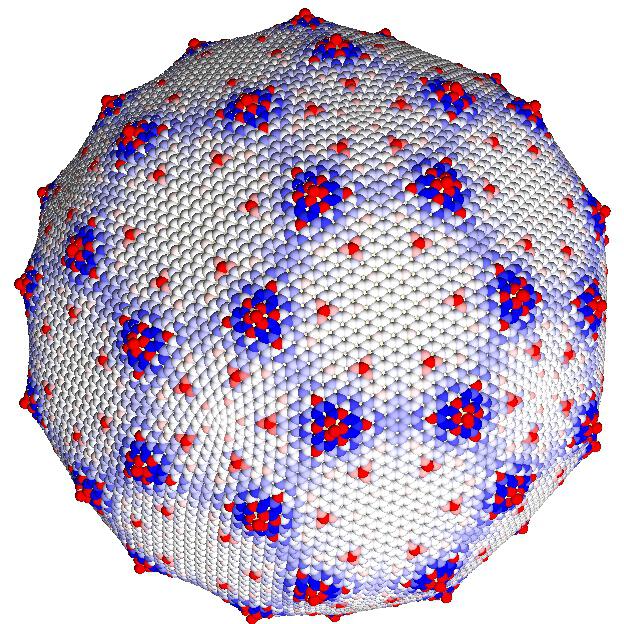}
        &\includegraphics[width=80pt]{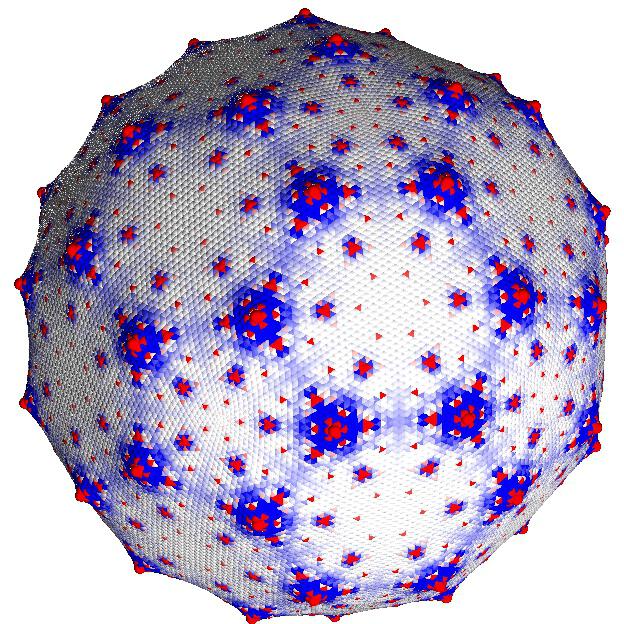}
          &\includegraphics[width=80pt]{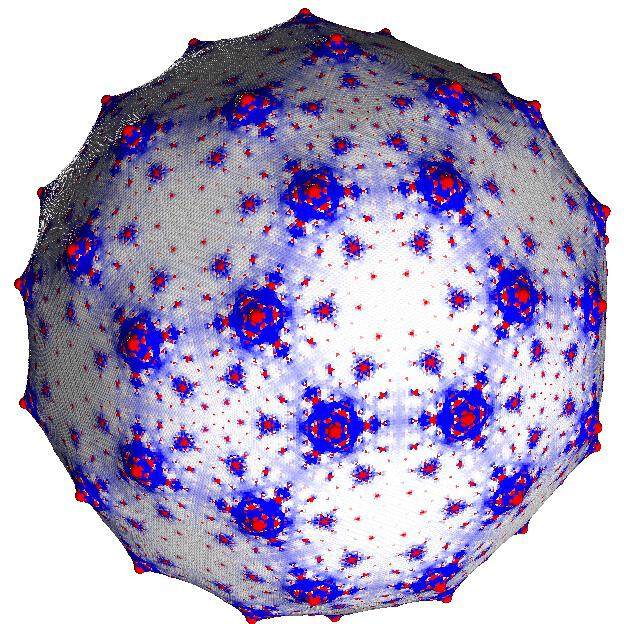}\\
    \hline
    \raisebox{35pt}{\parbox[c]{50pt}{modified\\Gauss curvature}}
    &\includegraphics[width=80pt]{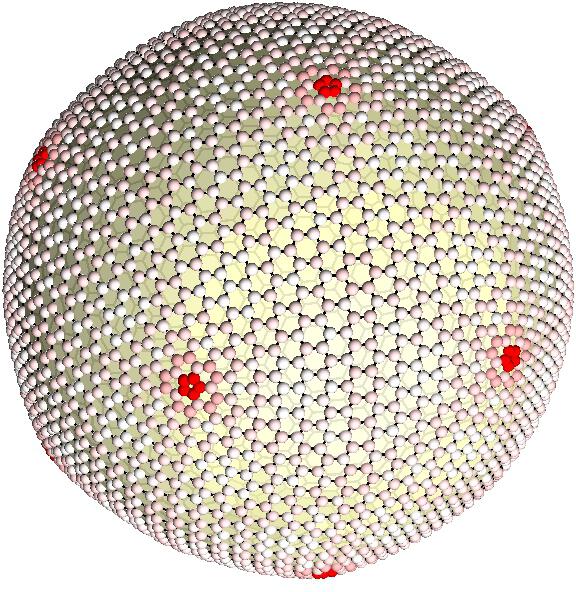}
      &\includegraphics[width=80pt]{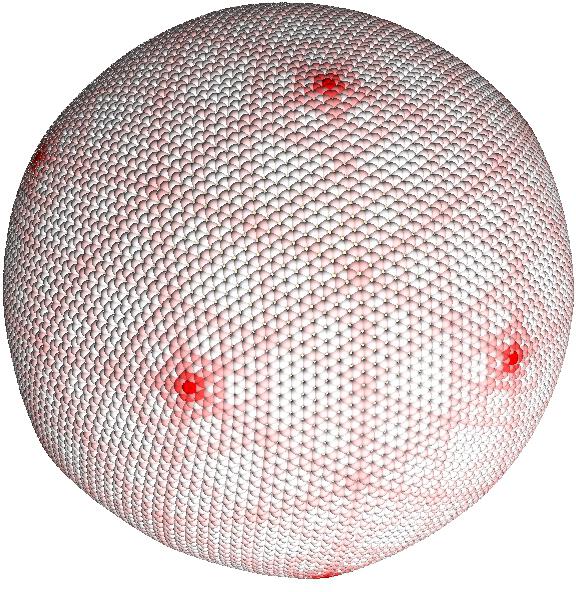}
        &\includegraphics[width=80pt]{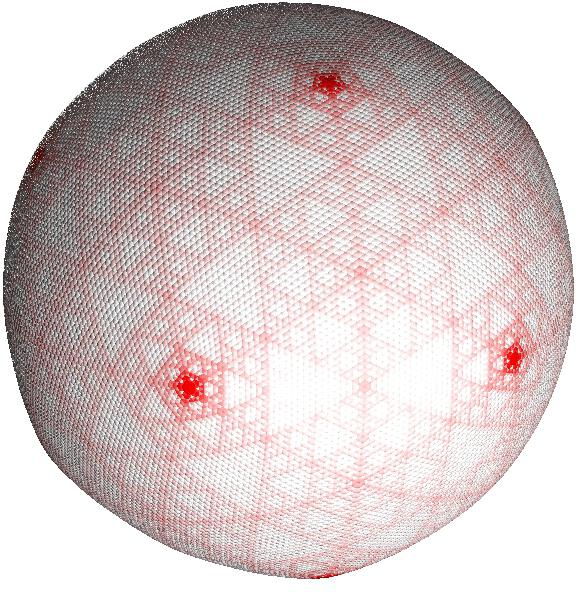}
          &\includegraphics[width=80pt]{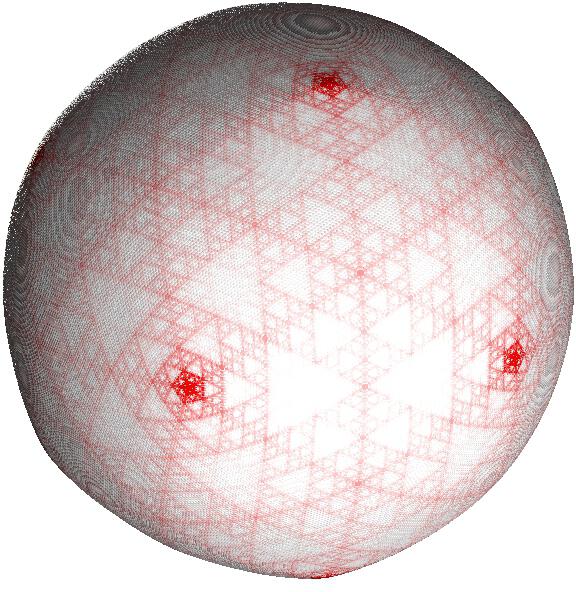}\\
    \hline
    \hline
    \raisebox{35pt}{\parbox[c]{50pt}{original\\mean curvature}}
    &\includegraphics[width=80pt]{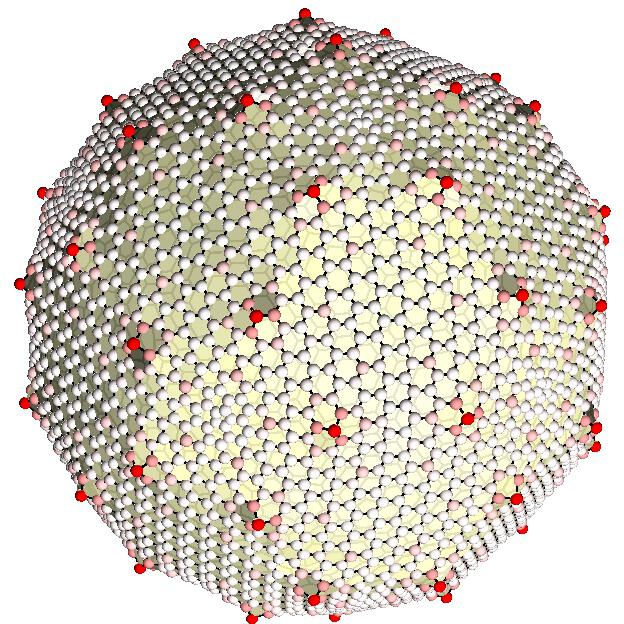}
      &\includegraphics[width=80pt]{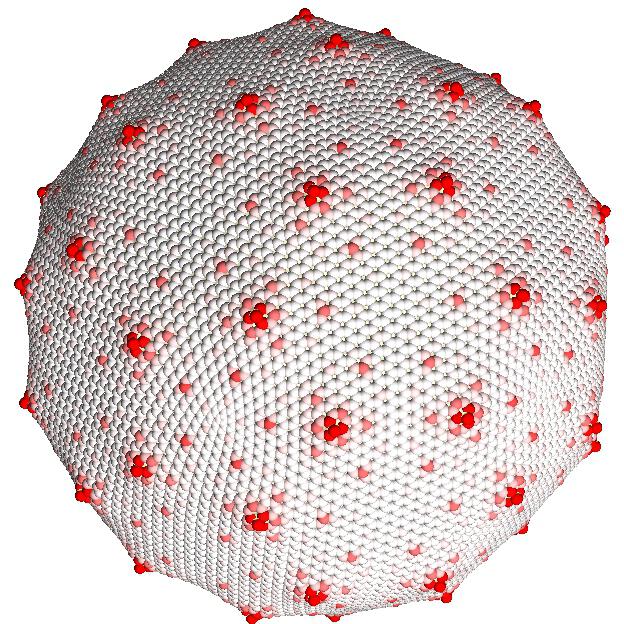}
        &\includegraphics[width=80pt]{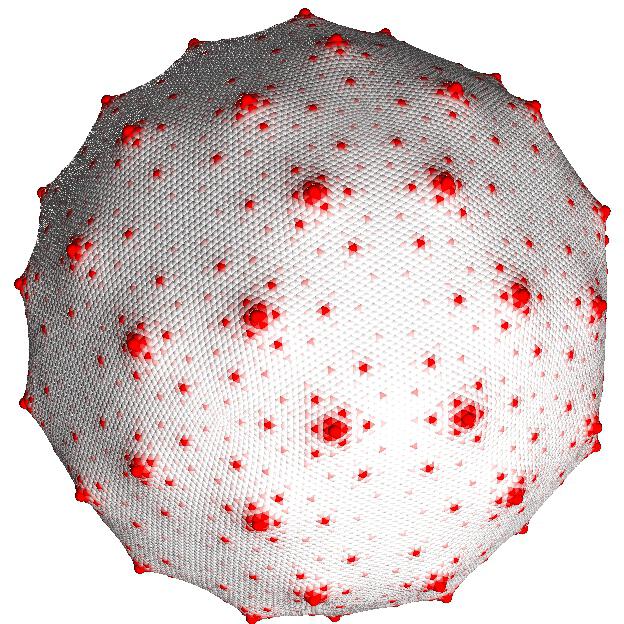}
          &\includegraphics[width=80pt]{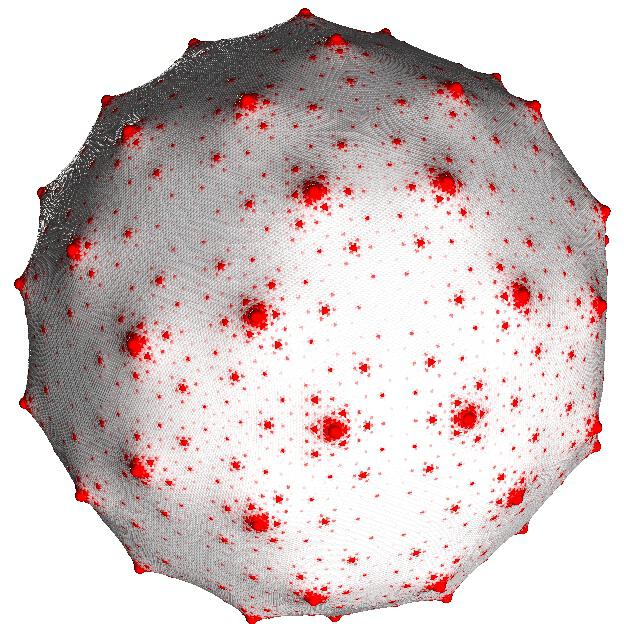}\\
    \hline
    \raisebox{35pt}{\parbox[c]{50pt}{modified\\mean curvature}}
    &\includegraphics[width=80pt]{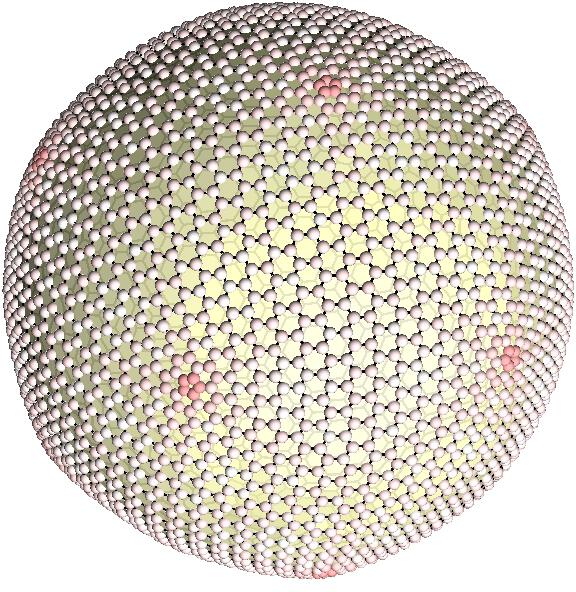}
      &\includegraphics[width=80pt]{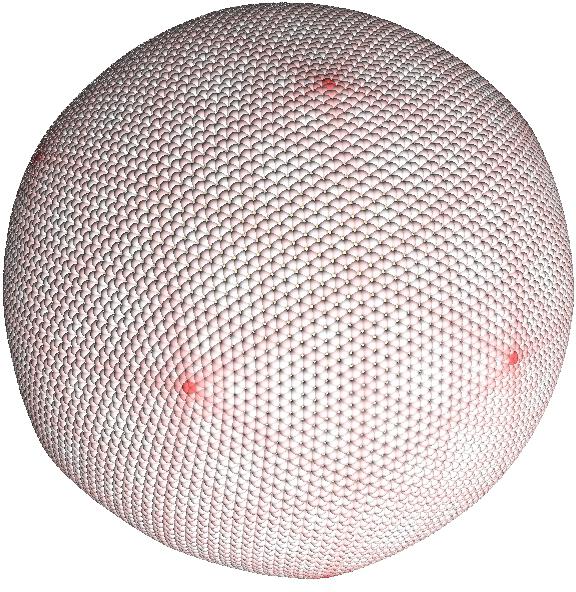}
        &\includegraphics[width=80pt]{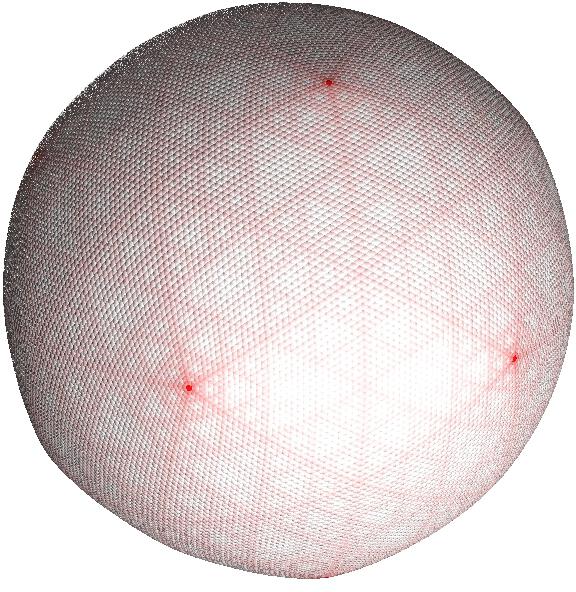}
          &\includegraphics[width=80pt]{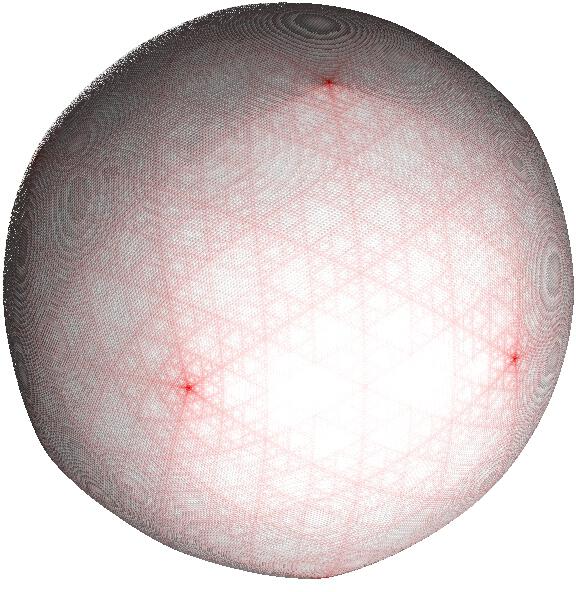}\\
    \mbox{} \\
    \multicolumn{5}{c}{\includegraphics[scale=0.50]{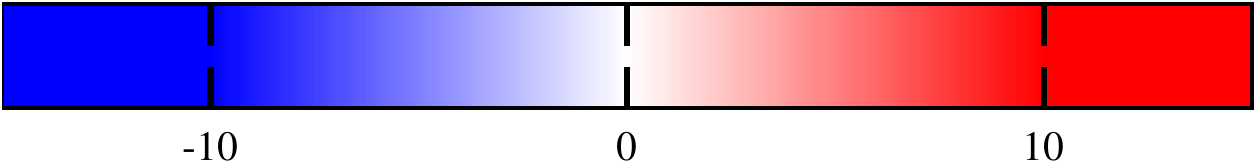}}
  \end{tabular}
  \caption{
    Gauss and mean curvatures of subdivisions for \ce{C60}.
  }
  \label{fig:c60:curvature}
\end{figure}
\begin{figure}[H]
  \centering
  \begin{tabular}{l|cccc}
    &3&4&5&6\\
    \hline
    \raisebox{35pt}{\parbox[c]{50pt}{original\\Gauss curvature}}
    &\includegraphics[width=80pt]{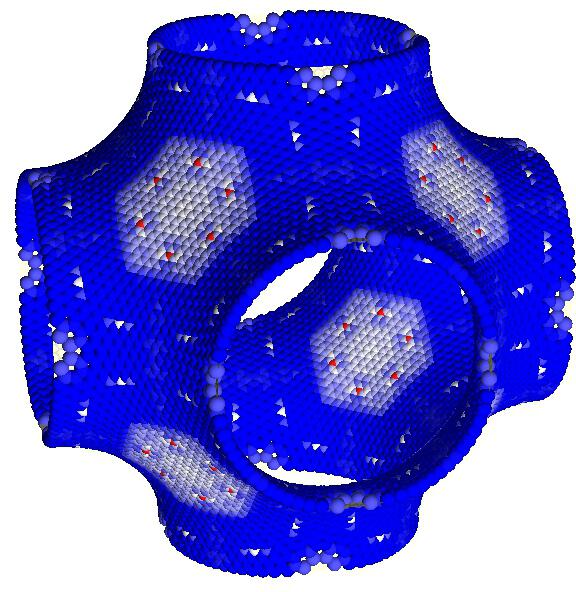}
      &\includegraphics[width=80pt]{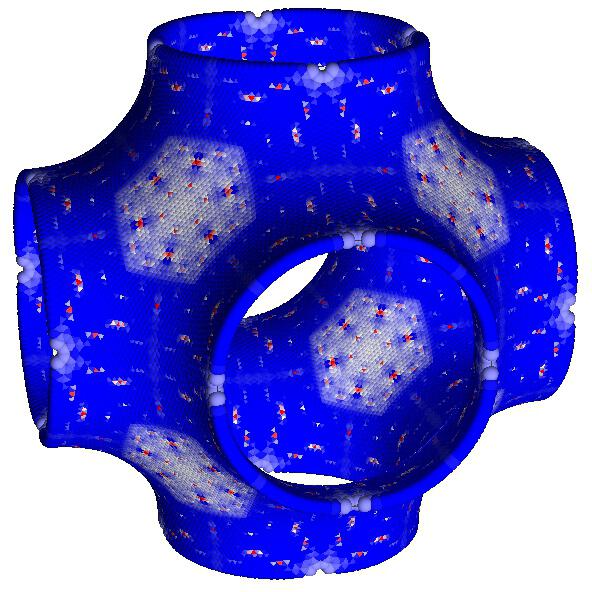}
        &\includegraphics[width=80pt]{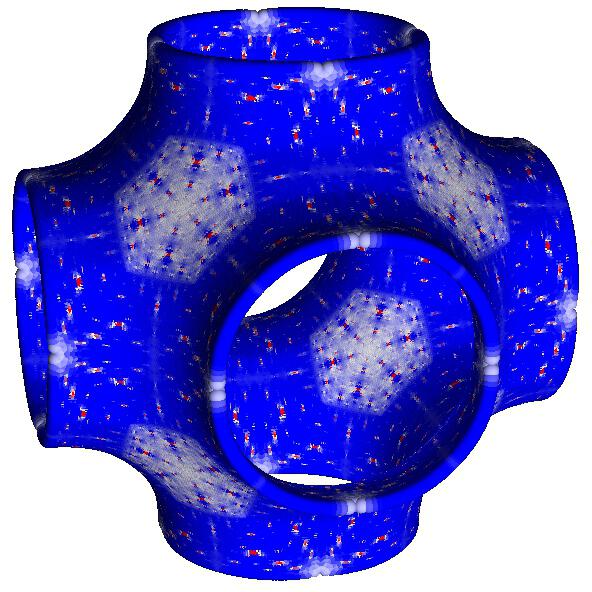}
          &\includegraphics[width=80pt]{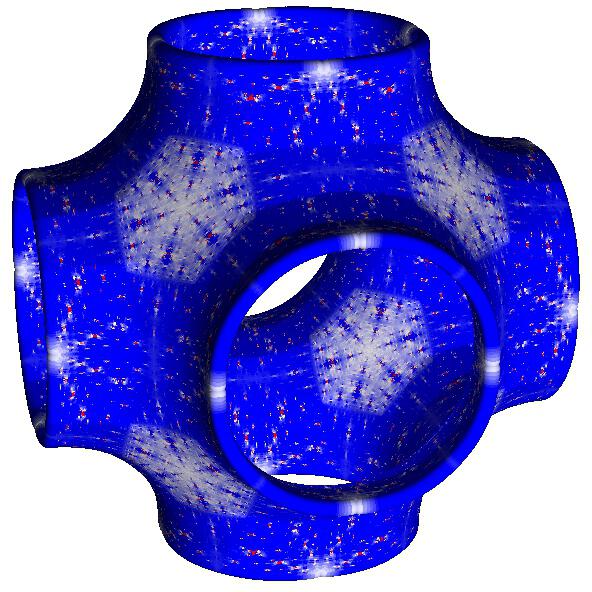}\\
    \hline
    \raisebox{35pt}{\parbox[c]{50pt}{modified\\Gauss curvature}}
    &\includegraphics[width=80pt]{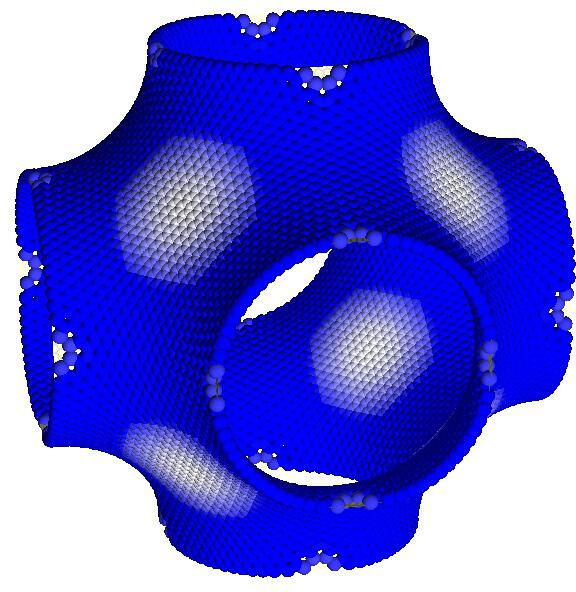}
      &\includegraphics[width=80pt]{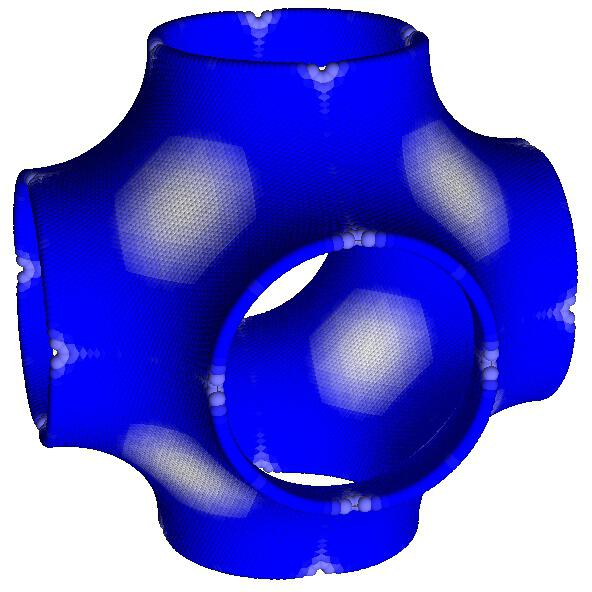}
        &\includegraphics[width=80pt]{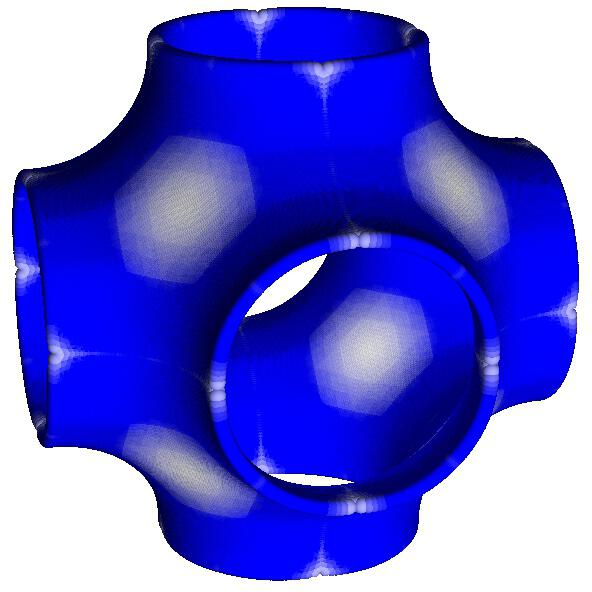}
          &\includegraphics[width=80pt]{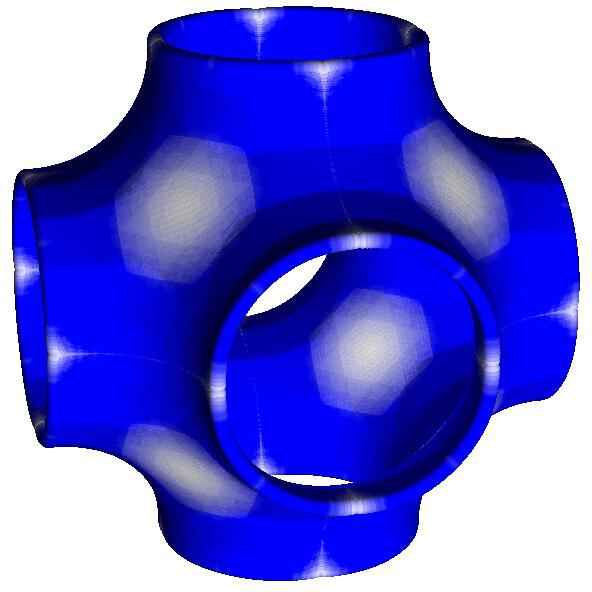}\\
    \hline
    \raisebox{35pt}{\parbox[c]{50pt}{original\\mean curvature}}
    &\includegraphics[width=80pt]{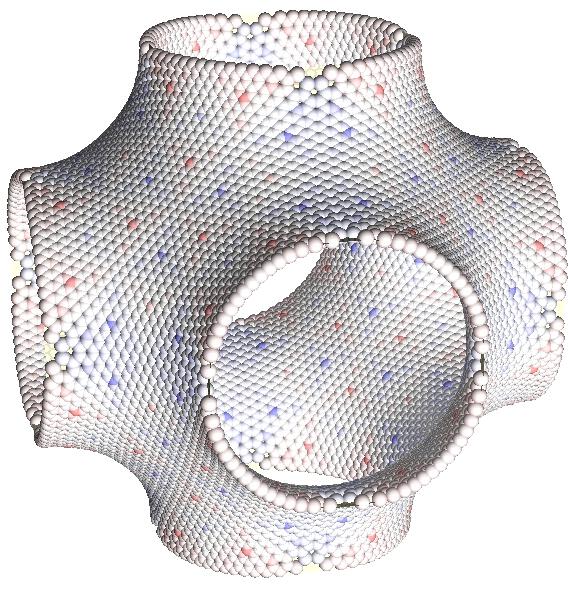}
      &\includegraphics[width=80pt]{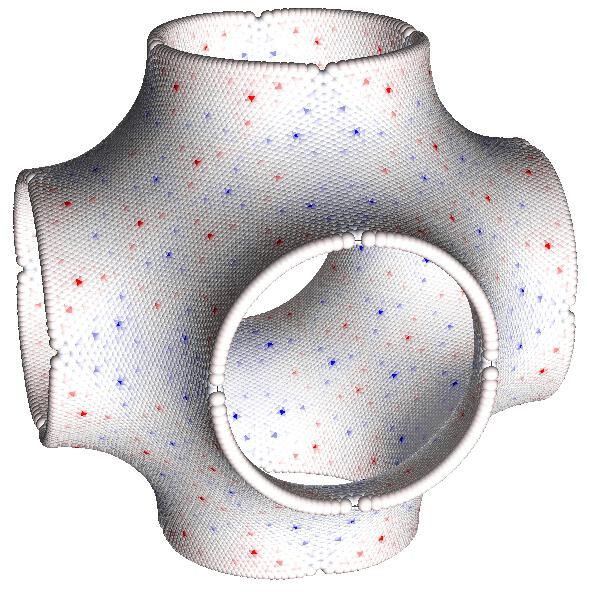}
        &\includegraphics[width=80pt]{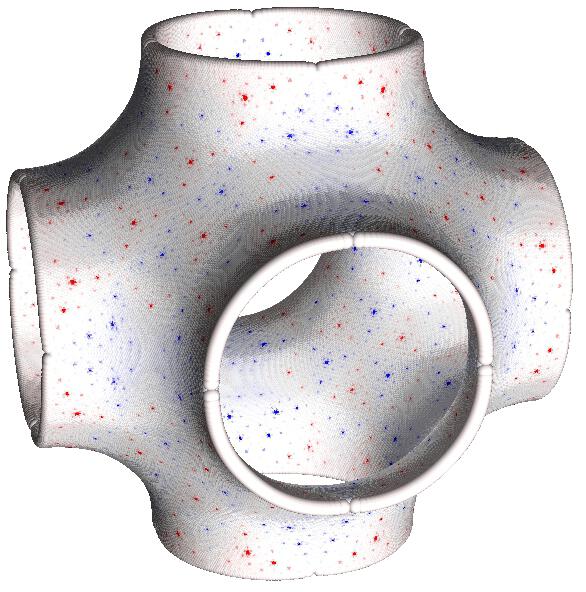}
          &\includegraphics[width=80pt]{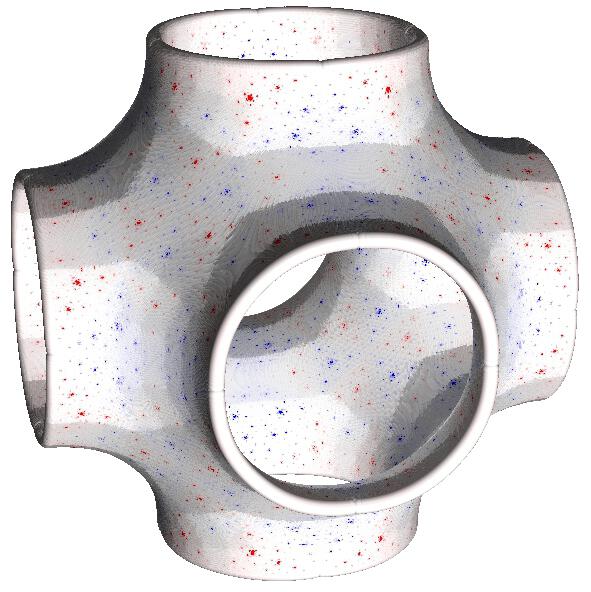}\\
    \hline
    \raisebox{35pt}{\parbox[c]{50pt}{modified\\mean curvature}}
    &\includegraphics[width=80pt]{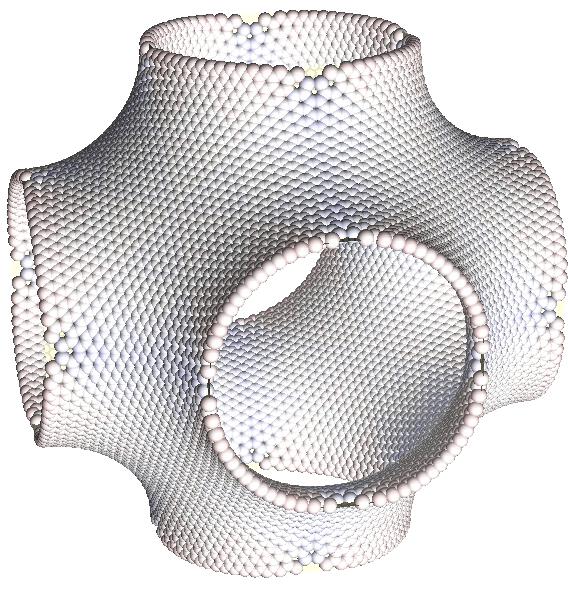}
      &\includegraphics[width=80pt]{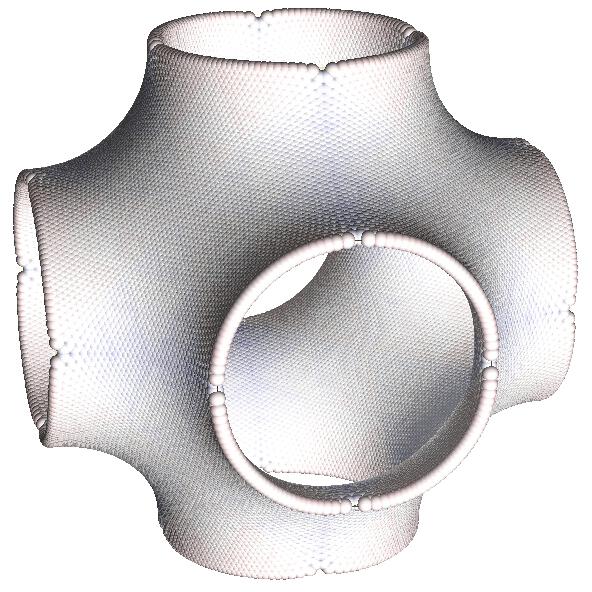}
        &\includegraphics[width=80pt]{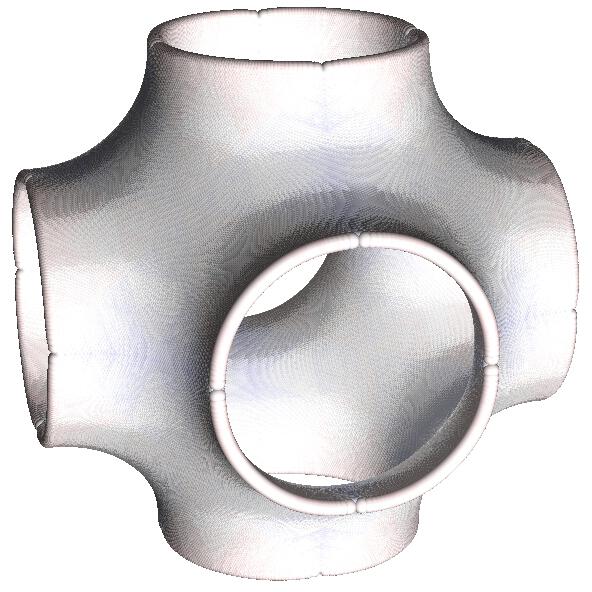}
          &\includegraphics[width=80pt]{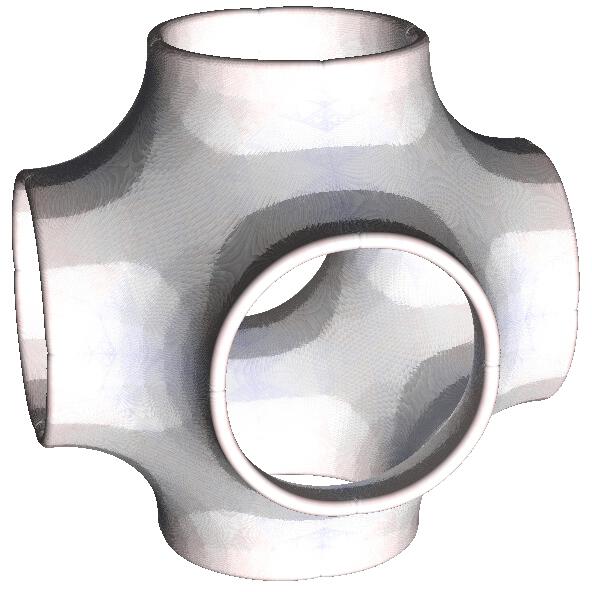}\\
    \mbox{} \\
    \multicolumn{5}{c}{\includegraphics[scale=0.50]{figures/C60-mean-crop.pdf}}
  \end{tabular}
  \caption{
    Gauss and mean curvatures of subdivisions for Mackay crystal.
  }
  \label{fig:mackay:curvature}
\end{figure}
{
  \small
  \begin{table}[H]
    \centering
    \begin{tabular}{l|r|r||r|r}
      \mbox{}&\multicolumn{4}{c}{Gauss curvature for \ce{C60} and its subdivisions}\\
      \mbox{}&\multicolumn{2}{c||}{original}&\multicolumn{2}{c}{modified}\\
      \mbox{}&\multicolumn{1}{c|}{min}&\multicolumn{1}{c||}{max}&\multicolumn{1}{c|}{min}&\multicolumn{1}{c}{max}\\
      \hline
      0&+1.000000000000&+1.000000000000&+1.000000000000&+1.000000000000\\
      1&+0.000000000000&+4.252873545461&+0.525777056286&+1.984363480318\\
      2&+0.000000000000&+23.710805396918&+0.121201828300&+4.635544466166\\
      3&+0.000000000000&+144.213543580835&+0.019088250650&+11.471725490635\\
      4&+0.000000000000&+796.106165831539&+0.004795736307&+29.519631805770\\
      5&+0.000000000000&+4158.979312625750&+0.001553719365&+77.844026602244\\
      6&+0.000000000002&+20183.584874378899&+0.001513497809&+208.379943586661\\
      \hline
      \hline
      \multicolumn{5}{c}{\mbox{}} \\
      \mbox{}&\multicolumn{4}{c}{Absolute values of mean curvature for \ce{C60} and its subdivisions}\\
      \mbox{}&\multicolumn{2}{c||}{original}&\multicolumn{2}{c}{modified}\\
      \mbox{}&\multicolumn{1}{c|}{min}&\multicolumn{1}{c||}{max}&\multicolumn{1}{c|}{min}&\multicolumn{1}{c}{max}\\
      \hline
      0&  +1.000000000000& +1.000000000000	& +1.000000000000& +1.000000000000 \\
      1&  +0.632214023392& +2.065746461116	& +0.736142329795& +1.445952675469 \\
      2&  +0.301253084257& +4.870995934069	& +0.378978641299& +2.226705188612 \\
      3&  +0.147387759935& +12.016186135011	& +0.208167970265& +3.522259721963 \\
      4&  +0.073257722208& +28.216464341791	& +0.140779306392& +5.670917251952 \\
      5&  +0.036573409537& +64.514723533582	& +0.086133053358& +9.230418205413 \\
      6&  +0.001203528409& +142.071641599498	& +0.057571796752& +15.123897080553 \\
    \end{tabular}
    \caption{
      Numerical results of maximum and minimum of Gauss and (absolute values of) mean curvatures 
      of subdivisions for \ce{C60}.
    }
    \label{table:c60:curvature}
  \end{table}
}
{
  \small
  \begin{table}[H]
    \centering
    \begin{tabular}{l|r|r||r|r}
      \mbox{}&\multicolumn{4}{c}{Gauss curvature for Mackay crystals and its subdivisions}\\
      \mbox{}&\multicolumn{2}{c||}{original}&\multicolumn{2}{c}{modified}\\
      \mbox{}&\multicolumn{1}{c|}{min}&\multicolumn{1}{c||}{max}&\multicolumn{1}{c|}{min}&\multicolumn{1}{c}{max}\\
      \hline
      0&  +3.771349862259& +17.666681446413	& +3.771349862259& +17.666681446413 \\
      1&  +0.000000000021& +14.397226459450	& +0.000000000021& +13.119488293420 \\
      2&  +0.000000000000& +16.656461806495	& +0.000000000071& +15.066544032947 \\
      3&  +0.000000000014& +23.082901682419	& +0.000000000142& +16.641060037832 \\
      4&  +0.000000000000& +85.143362076065	& +0.000000000330& +17.437385717505 \\
      5&  +0.000000000000& +601.064421328300	& +0.000000000016& +18.086560932206 \\
      6&  +0.000000000000& +4088.346117693280	& +0.000000001469& +18.572122871784 \\
      \hline
      \hline
      \multicolumn{5}{c}{\mbox{}} \\
      \mbox{}&\multicolumn{4}{c}{Absolute values of mean curvature for Mackay crystals and its subdivisions}\\
      \mbox{}&\multicolumn{2}{c||}{original}&\multicolumn{2}{c}{modified}\\
      \mbox{}&\multicolumn{1}{c|}{min}&\multicolumn{1}{c||}{max}&\multicolumn{1}{c|}{min}&\multicolumn{1}{c}{max}\\
      \hline
      0&  +0.029880403429& +0.586577778209	& +0.029880403429& +0.586577778209 \\
      1&  +0.063943067848& +0.802761087825	& +0.059318491199& +0.833555289165 \\
      2&  +0.014255956693& +1.673683654486	& +0.002132925185& +1.075287177741 \\
      3&  +0.000713636661& +3.949630889326	& +0.000262137958& +1.215025924198 \\
      4&  +0.001540074493& +9.607278290803	& +0.001724570423& +1.350175041239 \\
      5&  +0.000037916564& +24.956905972599	& +0.000114955116& +1.509305677441 \\
      6&  +0.000052056290& +64.679724464465	& +0.000116955793& +1.681630672170 \\
    \end{tabular}
    \caption{
      Numerical results of maximum and minimum of Gauss and (absolute values of) mean curvatures 
      of subdivisions for Mackay crystal.
    }
    \label{table:mackay:curvature}
  \end{table}
}




\end{document}